\documentclass[a4paper]{article}
\usepackage[margin=15mm]{geometry}
\usepackage[utf8]{inputenc}
\usepackage{graphics}
 \usepackage{a4wide}
\usepackage{epsfig}
\usepackage{xspace}
\usepackage{url}
\usepackage{amsmath,amscd,amsthm}
\usepackage{amsfonts}
\usepackage{multirow,booktabs}
\usepackage{amssymb}
\usepackage{hyperref,tikz,tikz-cd}
\usepackage{caption}
\usepackage{accents}
\usepackage{subcaption}
\usepackage{pdfpages}
\usepackage[mathscr]{euscript}
\usepackage{mathrsfs}
\usepackage{listings}
 \usepackage{enumerate}
 \usepackage{cite}
 \usepackage{hyperref}
 \hypersetup{linkcolor=blue, colorlinks=true ,citecolor = red}
\usetikzlibrary{decorations.pathreplacing,decorations.markings}

\newtheorem{theorem}{Theorem}[section]
\newtheorem{lemma}[theorem]{Lemma}

\newtheorem{corollary}[theorem]{Corollary}
\newtheorem{remark}[theorem]{Remark}
\newtheorem{definition}[theorem]{Definition}

\catcode`@=11 \@addtoreset{equation}{section}
\renewcommand\theequation{\thesection.\@arabic\c@equation}
\catcode`@=12

\makeatother
\usepackage[english]{babel}
\providecommand{\keywords}[1]
{
  \small	
  \textbf{\textit{Keywords---}} #1
}
\providecommand{\MSC}[1]
{
  \small	
  \textbf{\textit{Mathematics Subject Classification---}} #1
}
\usepackage{titlesec,authblk}

 \usepackage[hyperpageref]{backref}

\title{Degenerate Schr{\"o}dinger-Kirchhoff $(p, N)$-Laplacian problem With singular Trudinger-Moser nonlinearity in $\mathbb{R}^N$}
\date{}
\author{Deepak Kumar Mahanta, Tuhina Mukherjee$\thanks{Corresponding author}$ ~and Abhishek Sarkar  \\
        \small  Department of Mathematics, Indian Institute of Technology Jodhpur, Rajasthan 342030, India \\
        }
\newcommand{\Addresses}{{
  \bigskip
  \footnotesize

  D.K.~Mahanta, \textit{E-mail address:} \texttt{mahanta.1@iitj.ac.in}

  \medskip

  T.~Mukherjee, \textit{E-mail address:} \texttt{tuhina@iitj.ac.in}

  \medskip

  A.~Sarkar, \textit{E-mail address:} \texttt{abhisheks@iitj.ac.in}

}}

\begin{document}
\maketitle \vspace{-1.8\baselineskip}
\begin{abstract}
   In this paper, we deal with the existence of nontrivial nonnegative solutions for a $(p, N)$-Laplacian Schr{\"o}dinger-Kirchhoff problem in $\mathbb{R}^N$ with singular exponential nonlinearity. The main features of the paper are the $(p, N)$ growth of the elliptic operators, the double lack of compactness, and the fact that the Kirchhoff function is of degenerate type. To establish the existence results, we use the mountain pass theorem, the Ekeland variational principle, the singular Trudinger-Moser inequality, and a completely new Br\'ezis-Lieb type lemma for singular exponential nonlinearity.
\end{abstract}
\\
\noindent\MSC{35D30,35J60, 35J92}

\section{Introduction}
Given $N\geq 2$, we aim to investigate the existence of nontrivial nonnegative solutions for possibly degenerate  Schr{\"o}dinger-Kirchhoff $(p, N)$-Laplacian problem with singular Trudinger-Moser nonlinearity in $\mathbb{R}^N$. To be more specific, we will consider the following equation: 
\begin{equation}\label{main problem}
  \begin{cases}
  \mathcal{L}_{p,V}(u)+\mathcal{L}_{N,V}(u)=\lambda_1 h(x)|u|^{q-2}u+\lambda_2~\cfrac{g(x,u)}{|x|^\beta}~~\text{in}~~\mathbb{R}^N,\\  
 \displaystyle\int_{\mathbb{R}^N} V(x)\big(|u|^p+|u|^N\big)~\mathrm{d}x<\infty,~ u\in W^{1,p}(\mathbb{R}^N)\cap W^{1,N}(\mathbb{R}^N),
  \end{cases}\tag{$\boldsymbol{S}_{\lambda_1,\lambda_2}$}
\end{equation}
where $$\mathcal{L}_{r,V}(u)=M\big( \|u\|_{W^{1,r}_V}^r\big)\big[-\Delta_r u+V(x)|u|^{r-2}u\big]$$
with $$\|u\|_{W^{1,r}_V}=\bigg(\int_{\mathbb{R}^{N}}|\nabla u|^r~\mathrm{d}x+\int_{\mathbb{R}^{N}}V(x)|u|^r~\mathrm{d}x\bigg)^{\frac{1}{r}}~\text{for}~r\in\{p,N\}.$$ 
Also, we define $\Delta_r u=\text{div}(|\nabla u|^{r-2}\nabla u)$ with $r\in\{p,N\}$ as the usual $r$-Laplacian operator, $0\leq \beta<N$, $1<p,q<N<\infty$, $\lambda_1$ and $\lambda_2$ are two positive parameters, $h:\mathbb{R}^N\to (0,\infty)$ is a weight function in $L^\eta(\mathbb{R}^N)$ with $\eta=\frac{N}{N-q}$, $M:[0,\infty)\to [0,\infty)$ is a Kirchhoff function, and the scalar potential $V:\mathbb{R}^N\to \mathbb{R} $ is a continuous function. The nonlinearity $g:\mathbb{R}^N\times \mathbb{R}\to \mathbb{R} $ has critical exponential growth at infinity, i.e. there exists a positive constant $\alpha_0$ such that the following condition holds:
$$\lim_{|u|\to\infty} |g(x,u)| \text{exp}(-\alpha|u|^{\frac{N}{N-1}})=\begin{cases}
    0~~~~~\text{if}~\alpha>\alpha_0\\
    +\infty~\text{if}~\alpha<\alpha_0
\end{cases},~ \text{uniformly in} ~x\in \mathbb{R}^N.$$
Furthermore, we assume the following conditions on the Kirchhoff function $M$,  the scalar potential $V$ and the nonlinearity $g:\mathbb{R}^N\times \mathbb{R}\to \mathbb{R}$ :
\begin{itemize}
    \item[(M)] $M:[0,\infty)\to [0,\infty)$ is a degenerate Kirchhoff function, which is continuous and defined by $M(s)=as^k ~(a,k>0)$, with $q<p(k+1)$. Consequently, we define the primitive of $M$ as follows: $$\widehat{M}(t)=\int_{0}^{t}M(s)~\mathrm{d}s,~\forall~t>0.$$
    \item[(V1)] $V\in\mathcal{C}(\mathbb{R}^N,\mathbb{R})$ and there exists a constant $V_0>0$ such that $\inf_{x\in\mathbb{R}^N} V(x)\geq V_0$.
    \item[(V2)] $V(x)\to \infty$  as $|x|\to\infty$; or more generally, $\mathrm{meas}~\big(\{x\in \mathbb{R}^N :V(x)\leq c\}\big)<\infty,~\forall~ c>0$. The Lebesgue measure of $E\subset \mathbb{R}^N$ is denoted by $\mathrm{meas}(E)$. 
    \item[(G1)] $g$ is a Carath\'eodory function such that $g(x,u)=0$ for all $u\leq 0$ and for a.e. $x\in \mathbb{R}^N$. Moreover, there exist two constants $b_1,b_2>0$ and $0<\alpha_0<\alpha$ such that for all $(x,u)\in \mathbb{R}^N\times \mathbb{R}$, we have
    $$|g(x,u)|\leq b_1 |u|^{N(k+1)-1}+b_2 \Phi(\alpha_0 |u|^{N^\prime}),~\text{where}~\Phi(t)=\exp(t)-\sum_{j=0}^{N-2}\frac{t^j}{j!}~\text{and}~N^\prime=\frac{N}{N-1}.$$
    \item[(G2)] There exists $\mu>N(k+1)$ such that for all $x\in \mathbb{R}^N$ and $u>0$, we have
    $$0<\mu G(x,u)=\mu\int_{0}^{u}g(x,s)\mathrm{d}s\leq u g(x,u).$$
    \item[(G3)] The following hold
    $$\limsup_{t\to 0^+}\frac{N G(x,t)}{\widehat{M}(1)|t|^{N(k+1)}}< \mathcal{S}_{N,\beta} ~\text{uniformly for}~ x\in \mathbb{R}^N,~\text{where}~\mathcal{S}_{N,\beta}~\text{is defined in}~\eqref{eq2.1}.$$ 
    \item[(G4)] $\partial_u g$ is a Carath\'eodory function such that $\partial_u g(x,u)=0$ for all $u\leq 0$ and for a.e. $x\in \mathbb{R}^N$. Consequently, there exist two constants $b_1,b_2>0$ and $0<\alpha_0<\alpha$ such that for all $(x,u)\in \mathbb{R}^N\times \mathbb{R}$, we have
    $$|\partial_u g(x,u)u|\leq b_1 |u|^{N(k+1)-1}+b_2 \Phi(\alpha_0 |u|^{N^\prime}).$$
    \item[(G5)] There exists $\xi>N(k+1)$ and $C>0$ such that $$G(x,u)\geq \frac{2C}{\xi}u^\xi,~\text{for all}~(x,u)\in \mathbb{R}^N\times [0,\infty). $$
\end{itemize}

Recent literature is flooded with $(p,q)$ type or in general, double phase problems but the borderline case, when $q=N$ and $p<q$, is comparatively less looked upon, which is the key motivation towards this article.
It is worthwhile to mention that for $p<N$,  the  Sobolev embedding $W^{1,p}(\mathbb{R}^N)\hookrightarrow L^{p^\ast}(\mathbb{R}^N)$ holds, where $p^\ast=\frac{Np}{N-p}$ if $p<N$ and $p^\ast=+\infty$ if $p\geq N$. In this situation, the variational technique is used to solve most of the problems in which the nonlinearity exhibits polynomial growth at infinity. On the other hand, in the borderline case, i.e., $p=N$, the embedding $W^{1,N}(\mathbb{R}^N)\hookrightarrow L^{\theta}(\mathbb{R}^N)$ holds for any $N\leq \theta<\infty$; however $W^{1,N}(\mathbb{R}^N)$ is not contained in $ L^{\infty}(\mathbb{R}^N)$. Consequently, every polynomial growth is allowed in this instance. As a result, we look for a function $f:\mathbb{R}\to (0,\infty)$ with maximal growth satisfying the following condition:
\begin{align*}
\sup_{u\in W^{1,N}(\mathbb{R}^N), \|u\|_{W^{1,N}}\leq 1} ~\displaystyle\int_{\mathbb{R}^N}\cfrac{f(u)}{|x|^\beta}~\mathrm{d}x<+\infty,~\text{for all}~~0\leq\beta<N.    
\end{align*}
It has been seen by many authors that this maximal growth is of exponential type  ( see \cite{MR2669653, MR2788328, MR3042696, MR2838280, MR2198572}). Hence, there has been considerable interest paid to study $N$-Laplacian problems in the sense of Trudinger-Moser inequality.

In last few decades, the results of Trudinger-Moser have been studied extensively in several directions. In that context, we want to mention \cite{MR2669653, MR3836136, MR2863858} and references for singular Trudinger-Moser inequality in $\mathbb{R}^N$. On the other hand, one can see to \cite{do1997n, MR2488689, MR3145759, MR2400264} and references therein for sharp Trudinger-Moser inequality in $\mathbb{R}^N$. Further, we refer the interested readers to study \cite{MR2980499, MR3130507, MR3797740} for sharp Trudinger-Moser inequalities in the Heisenberg group.
We also refer to the articles \cite{MR1230930, MR1989228} as applications of Trudinger-Moser inequalities to the study of extremal problems for determinants and zeta functions under conformal deformation metric and in the theory of conformal geometry and geometric analysis.

In a mathematical point of view, elliptic problems driven by $(p, N)$-Laplacian operator have several technical difficulties due to the fact that this operator is not homogeneous, the double lack of compactness and the Palais-Smale sequences do not have the compactness property. In order to overcome the lack of compactness which often occurs in the problems associated with the unbounded domain we assume some hypothesis on the potential. To the best of our knowledge, in the local framework, there are very few contributions in the literature involving $(p, N)$-Laplacian operator with Trudinger-Moser nonlinearity in $\mathbb{R}^N$. For instance, the authors of \cite{MR4258779} established the existence of nonnegative solutions for a $(p, N)$-Laplacian equation with critical exponential nonlinearities in $\mathbb{R}^N$ via mountain pass theorem and Ekeland variational principle. In \cite{MR4149293}, the authors studied the existence of at least one nonnegative solution for a coupled exponential system in the entire space driven by a $(p, N)$-Laplacian operator via suitable variational and topological methods. By applying variational technique, the authors of \cite{MR4289110} proved a completely new Trudinger-Moser type inequality. Then, they obtained a nonzero weak solution for the $(N, q)$-Laplacian equation with critical exponential nonlinearity in $\mathbb{R}^N$. Moreover, the problems driven by the $(p, N)$-Laplacian operator on bounded domains with critical exponential growth in the sense of Trudinger-Moser can be seen in the papers \cite{MR3483063, MR3896668, MR4365176}.

In recent years, it is worth mentioning that the study of Kirchhoff-type problems, which was first proposed by Kirchhoff in \cite{kirchhoff1883vorlesungen}, has been paid considerable attention by many authors due to its application in various models of physical and biological systems. A typical example of the Kirchhoff function $M$ is given by $M(t)=a+bt^k$ for $t\geq0,k>0$, where $a,b\geq 0$ and $a+b>0$. The problem involving Kirchhoff function $M$ is called degenerate if $a=0$ while it is called non-degenerate if $a>0$. Subsequently, it is important to note that the degenerate case in Kirchhoff's theory is more interesting and significant than the non-degenerate case. However, we would like to mention that for both local and nonlocal frameworks, most of the literature deals with non-degenerate Kirchhoff problems with Trudinger-Moser nonlinearities in $\mathbb{R}^N$, it can be found in  \cite{MR3987388, MR4334134, MR4622428, MR3735478, MR4487653} and for degenerate type Kirchhoff problem with critical exponential growth in $\mathbb{R}^N$ there are few articles (see \cite{MR3316612, MR4596285}). 

The main novelty of our paper is the presence of two Kirchhoff coefficients and the fact that the Kirchhoff function is of a degenerate type. Due to the presence of this degenerate Kirchhoff term $M$ in \eqref{main problem}, it is not trivial to show that the weak limit of a bounded sequence in the desired function space  $\mathbf{X}$, which is defined in \eqref{eq1}, is a solution to \eqref{main problem}. To overcome this difficulty, by adapting the idea of the proof of Lemma.$6$ in \cite{MR4411714}, we first provide a compactness result for the map $\mathscr{L}^{\prime}: \mathbf{X}\to \mathbf{X}^\ast$ (see Lemma \ref{lem2.12}), where $\mathbf{X}^\ast$ denotes the continuous dual of $\mathbf{X}$.

Furthermore, no results for degenerate Kirchhoff problems driven by $(p, N)$-Laplacian operator with singular exponential nonlinearity in $\mathbb{R}^N$ appear in the current literature. Motivated by this fact and the results obtained by Fiscella and Pucci in \cite{MR4258779}, we study the existence of solutions to a degenerate $(p, N)$-Laplacian Schr{\"o}dinger- Kirchhoff problem with singular exponential growth in $\mathbb{R}^N$.

Let $r\in(1,\infty)$ and $L^r(\mathbb{R}^N)$ denotes the standard Lebesgue space w.r.t the norm $\|\cdot\|_r$. Further, if $\Omega\subset \mathbb{R}^N$, then we define the norm by $\|\cdot\|_{L^r(\Omega)}$. Now for nonnegative measurable function $V:\mathbb{R}^N\to \mathbb{R}$, the space $L^r_V(\mathbb{R}^N)$, consisting of all real-valued measurable functions, with $V(x)|u|^r\in L^1(\mathbb{R}^N)$, equipped with the seminorm
$$ \|u\|_{r,V}=\Bigg(\int_{\mathbb{R}^N} V(x)|u|^r~\mathrm{d}x\Bigg)^{\frac{1}{r}},~\text{for all}~u\in L^r_V(\mathbb{R}^N), $$
which is a norm if $V(x)>0$ a.e. $x\in \mathbb{R}^N$. The space $\big(L^r_V(\mathbb{R}^N),\|\cdot\|_{r,V}\big)$ is a  separable and uniform convex Banach space (see \cite{MR3412392}), thanks to $(V1)$. Moreover, we note that the embedding $L^r_V(\mathbb{R}^N)\hookrightarrow L^r(\mathbb{R}^N)$ is continuous.

For $0\leq\beta<N$, we define the space $L^r(\mathbb{R}^N,|x|^{-\beta})$, consisting of all real-valued measurable functions, with $\frac{|u|^r}{|x|^{\beta}}\in L^1(\mathbb{R}^N)$, equipped with the norm
$$ \|u\|_{L^r(\mathbb{R}^N,|x|^{-\beta})}=\Bigg(\int_{\mathbb{R}^N} \frac{|u|^r}{|x|^{\beta}}~\mathrm{d}x\Bigg)^{\frac{1}{r}},~\text{for all}~u\in L^r(\mathbb{R}^N,|x|^{-\beta}).$$ 

Denote $$ W^{1,r}(\mathbb{R}^N)=\big\{ u\in L^{r}(\mathbb{R}^N):|\nabla u|\in L^{r}(\mathbb{R}^N)\big\}, $$
endowed with the norm
$$\|u\|_{W^{1,r}}=\big(\|\nabla u\|_r^r+\|u\|_{r}^r\big)^{\frac{1}{r}}.$$
It is well-known that the space $\big(W^{1,r}(\mathbb{R}^N),\|\cdot\|_{W^{1,r}}\big)$ is a separable and uniformly convex Banach space. Note that, the space of smooth function with compact support which is denoted by $\mathcal{C}_0^\infty(R^N)$ is a dense subset of $W^{1,r}(\mathbb{R}^N)$.

Under the assumption of $(V1)$, the weighted Sobolev space $W^{1,r}_V(\mathbb{R}^N)$ make sense and defined by
$$W^{1,r}_V(\mathbb{R}^N)=\bigg\{ u\in W^{1,r}(\mathbb{R}^N):\int_{\mathbb{R}^N}V(x)|u|^r~\mathrm{d}x<\infty\bigg\}, $$
endowed with the norm
$$\|u\|_{W^{1,r}_V}=\bigg(\|\nabla u\|_r^r+\|u\|_{r,V}^r\bigg)^{\frac{1}{r}}.$$
The space $\big(W^{1,r}_V(\mathbb{R}^N),\|\cdot\|_{W^{1,r}_V}\big)$ is a separable and uniformly convex Banach space (see Proposition 2.1 in \cite{MR3454625}). Moreover, $\mathcal{C}_0^\infty(R^N)$ is a dense subset of $W^{1,r}_V(\mathbb{R}^N)$ (see \cite{MR3483510,MR3454625}). From now onward, we take into account the Banach space
\begin{equation}\label{eq1}
 \mathbf{X}=W^{1,p}_V(\mathbb{R}^N)\cap W^{1,N}_V (\mathbb{R}^N)   
\end{equation}
as our working space and endowed with the norm
$ \|u\|_{\mathbf{X}}=\|u\|_{W^{1,p}_V}+\|u\|_{W^{1,N}_V}. $
It is easy to see that the function space $(\mathbf{X},\|\cdot\|_{\mathbf{X}})$ is a reflexive and separable Banach space, thanks to $(V1)$, Proposition 2.1 in \cite{MR3454625} and Corollary 1.8 in \cite{MR0273370}.
\begin{definition}\label{definition}
We say $u\in\mathbf{X}$ is a (weak) solution of \eqref{main problem}, if
 $$ M\big( \|u\|_{W^{1,p}_V}^p\big)\big\langle{u,\psi}\big\rangle_{p,V}+M\big( \|u\|_{W^{1,N}_V}^N\big)\big\langle{u,\psi}\big\rangle_{N,V}= \lambda_1\int_{\mathbb{R}^N}h(x)|u|^{q-2}u\psi~\mathrm{d}x+\lambda_2\int_{\mathbb{R}^N}\cfrac{g(x,u)}{|x|^\beta}\psi~\mathrm{d}x $$ 
for all $\psi\in\mathbf{X} $, where  $\big\langle{u,\psi}\big\rangle_{r,V}$ for $r\in\{p,N\}$ is defined as follows:
\begin{align*}
\big\langle{u,\psi}\big\rangle_{r,V}=\int_{\mathbb{R}^N}|\nabla u|^{r-2}\nabla u.\nabla\psi~\mathrm{d}x+\int_{\mathbb{R}^N}V(x)|u|^{r-2}u\psi~\mathrm{d}x,~\text{for all}~u,\psi\in \mathbf{X} .   
\end{align*} 
\end{definition}

Now, we can state the existence results of (weak) solutions for \eqref{main problem}.

\begin{theorem}\label{thm1.1}
 Let $1<p,q<N<\infty$ and $h$ is in $L^\eta(\mathbb{R}^N)$ with $\eta=\frac{N}{N-q}$. Suppose $(M)$, $(V1)$, $(V2)$, $(G1)$-$(G3)$ are hold, then there exists $\widehat{\lambda}_1>0$ such that \eqref{main problem} admits at least one nontrivial nonnegative solution $u_{\lambda_1,\lambda_2}$ in $\mathbf{X}$ for all $\lambda_1\in (0,\widehat{\lambda}_1)$ and for all $\lambda_2>0$. Furthermore, $$\lim_{\lambda_1\to 0^+}\|u_{\lambda_1,\lambda_2}\|_{\mathbf{X}}=0. $$
\end{theorem}
\begin{theorem}\label{thm1.2}
Suppose that all the assumptions of Theorem \ref{thm1.1} hold. In addition, if $(G4)$ and $(G5)$ are satisfied, then there exists $\lambda_2^\ast>0$ such that for all $\lambda_2>\lambda_2^\ast$ there exists $\widetilde{\lambda}_1=\widetilde{\lambda}_1(\lambda_2)>0$ with the property that \eqref{main problem} admits a nontrivial nonnegative solution $v_{\lambda_1,\lambda_2}$ for all $\lambda_1\in (0,\widetilde{\lambda}_1]$. Consequently, if $\lambda_1<\textit{min}\{\widehat{\lambda}_1,\widetilde{\lambda}_1\}$, then $v_{\lambda_1,\lambda_2}$ is a second independent solution of \eqref{main problem} which is independent of $u_{\lambda_1,\lambda_2}$ constructed in Theorem \ref{thm1.1}.  
\end{theorem}

The outline of the paper is organized as follows. In Section \ref{sec2}, we present the variational structure of \eqref{main problem} and some preliminary results useful for the next main sections. Section \ref{sec3} is devoted to proving Theorem \ref{thm1.1} via the Ekeland variational principle and using topological tools. Finally, in Section \ref{sec4}, we establish the key compactness lemma, particularly helpful to apply the mountain pass theorem at a suitable mountain pass level and to prove Theorems \ref{thm1.2}.\\ 
\textbf{Notations.} Throughout this paper, we use the following notations:
\begin{itemize}
    \item [$\bullet$] $r^{\prime}=\frac{r}{r-1}$ is the conjugate exponent of $r>1$, that is, $\frac{1}{r}+\frac{1}{r^\prime}=1$.
    \item [$\bullet$] $o_n(1)$ denotes $o_n(1)\to 0$ as $n\to\infty$.
    \item [$\bullet$] $C_1,C_2,C_3,\dots$ denote some fixed constants possibly different in different places.
    \item [$\bullet$] $\rightharpoonup $ means weak convergence and $\rightarrow$ means strong convergence.
    \item [$\bullet$] $u^+=$ max~$\{u,0\}$ and $u^-=$ max~$\{-u,0\}$.
    \item [$\bullet$] $B_R=\{u\in \mathbf{X}:\|u\|_{\mathbf{X}}<R\}$, $\overline{B}_R=\{u\in \mathbf{X}:\|u\|_{\mathbf{X}}\leq R\}$, $\partial B_R=\{u\in \mathbf{X}:\|u\|_{\mathbf{X}}=R\}$.
    \item [$\bullet$] $\mathcal{S}_{r,\beta},~ \mathcal{S}_{r,h}$ and $\mathcal{S}_{r}$ are the best constants in the embeddings from $\mathbf{X}$ into $L^r(\mathbb{R}^N,|x|^{-\beta})$, $L^r_h(\mathbb{R}^N) $ and $L^r(\mathbb{R}^N)$ respectively.
    \item [$\bullet$] $\mathbb{R}^+=(0,\infty) $ and $\mathbb{R}_0^+=[0,\infty) $.
\end{itemize}

\section{Some Preliminary Results}\label{sec2}
In this section, we first provide some technical lemmas which will be used to prove our main results in this paper. The following first four embeddings are direct consequence of the classical embedding results in Sobolev space theory.
\begin{lemma}\label{lem2.1}
Let $(V1)$ holds. If $q\in[p,p^\ast]$, then the following embeddings are continuous $$ W^{1,p}_V(\mathbb{R}^N)\hookrightarrow W^{1,p}(\mathbb{R}^N) \hookrightarrow L^q (\mathbb{R}^N)$$ with $\text{min} \{1,V_0\} \|u\|^p_{W^{1,p}}\leq \|u\|^p_{W^{1,p}_V}$, for all $u\in W^{1,p}_V(\mathbb{R}^N)$. 
Furthermore, the embedding $W^{1,p}_V(\mathbb{R}^N)\hookrightarrow L^q _{\text{loc}}(\mathbb{R}^N)$ is compact for any $q\in[1,p^\ast)$.
\end{lemma}
\begin{lemma}\label{lem2.2}
Let $(V1)$ holds. If $s\in[N,\infty)$, then the following embeddings are continuous $$ W^{1,N}_V(\mathbb{R}^N)\hookrightarrow W^{1,N}(\mathbb{R}^N) \hookrightarrow L^s (\mathbb{R}^N)$$
with $\text{min}\{1,V_0\} \|u\|^N_{W^{1,N}}\leq \|u\|^N_{W^{1,N}_V}$, for all $u\in W^{1,N}_V(\mathbb{R}^N)$.
Furthermore, the embedding $W^{1,N}_V(\mathbb{R}^N)\hookrightarrow L^s _{\text{loc}}(\mathbb{R}^N)$ is compact for any $s\in[1,\infty)$.
\end{lemma}
\begin{corollary}\label{cor2.3}
 Let $(V1)$ holds, in view of Lemma \ref{lem2.1} and Lemma \ref{lem2.2}, the  embeddings
 $$\mathbf{X}\hookrightarrow W^{1,p}(\mathbb{R}^N)\cap W^{1,N}(\mathbb{R}^N) \hookrightarrow L^\vartheta (\mathbb{R}^N). $$
are continuous for any $\vartheta \in [p,p^\ast]\cup [N,\infty)$. Also, the embedding $\mathbf{X}\hookrightarrow L^\vartheta _{\text{loc}}(\mathbb{R}^N)$ is compact for any $\vartheta\in[1,\infty)$.
\end{corollary}
\begin{remark}\label{remark2.4}
    Under the assumption of $(V1)$, the following continuous embeddings hold:
     $$\mathbf{X}\hookrightarrow W^{1,r}_V(\mathbb{R}^N)\hookrightarrow W^{1,r}(\mathbb{R}^N)~\text{for}~r\in\{p,N\}.$$    
\end{remark}

\begin{lemma} [cf. \cite{MR2669653, MR2794422}]\label{lem2.5}
 Let $(V1)$ and $(V2)$ be satisfied, then the following results hold.
 \begin{itemize}
     \item[(a)] The embedding $W^{1,p}_V(\mathbb{R}^N)\hookrightarrow L^q (\mathbb{R}^N)$ is compact for any $q\in[p,p^\ast)$.
     \item[(b)] The  embedding $W^{1,N}_V(\mathbb{R}^N)\hookrightarrow L^s (\mathbb{R}^N)$ is compact for any $s\in[N,\infty)$.
 \end{itemize}
\end{lemma}
\begin{corollary}\label{cor2.6}
    Let $(V1)$ and $(V2)$ are hold true, then in view of Lemma \ref{lem2.5}, the embedding 
     $\mathbf{X}\hookrightarrow L^\tau(\mathbb{R}^N)$ is compact for any $\tau\in[p,p^\ast)\cup[N,\infty)$.   
\end{corollary}
\begin{lemma}[cf. \cite{MR4258779}]\label{lem2.7}
  The embedding $L^N(\mathbb{R}^N)\hookrightarrow L^q_h(\mathbb{R}^N)$ is continuous for any $q\in(1,N)$ and $$\|u\|_{q,h}\leq \|h\|^{\frac{1}{q}}_\eta \|u\|_N,~\text{for all}~u\in L^N(\mathbb{R}^N). $$
  Also, the embedding $W^{1,N}(\mathbb{R}^N)\hookrightarrow L^q_h (\mathbb{R}^N)$ is compact. Moreover, in view of remark \ref{remark2.4}, the following embeddings are compact $$W^{1,N}_V(\mathbb{R}^N)\hookrightarrow \hookrightarrow L^q_h (\mathbb{R}^N) \text{ and } \mathbf{X} \hookrightarrow\hookrightarrow L^q_h (\mathbb{R}^N).$$
\end{lemma}
The proof of the following lemma is a direct consequence of Lemma \ref{lem2.2}, Lemma \ref{lem2.5}, Theorem 1.2 in \cite{MR3836136} and remark \ref{remark2.4}. 
\begin{lemma}\label{lem2.8}
 Let $(V1)$ and $(V2)$ are satisfied. Suppose $q\geq N$ and $0\leq \beta <N$, then the embedding $$W^{1,N}_V(\mathbb{R}^N)\hookrightarrow L^q(\mathbb{R}^N,|x|^{-\beta})~\text{is compact}.$$ In particular, the embedding $\mathbf{X}\hookrightarrow L^q(\mathbb{R}^N,|x|^{-\beta})$ is compact for all $q\geq N$ and $0\leq \beta <N$. 
\end{lemma}

Define the best constant $\mathcal{S}_{N,\beta}$ in the embedding  $\mathbf{X}\hookrightarrow L^N(\mathbb{R}^N,|x|^{-\beta})$ as follows:
 
\begin{equation}\label{eq2.1}
   \mathcal{S}_{N,\beta}=\inf_{u\in\mathbf{X}\setminus\{0\}}\frac{\|u\|_{\mathbf{X}}}{\|u\|_{L^N(\mathbb{R}^N,|x|^{-\beta})}} >0.
\end{equation}
Clearly, $$\|u\|_{L^N(\mathbb{R}^N,|x|^{-\beta})}\leq \mathcal{S}^{-1}_{N,\beta}\|u\|_{\mathbf{X}},~\text{for all}~u\in\mathbf{X}.$$
\begin{corollary} [cf. \cite{MR2873855}]\label{cor2.9}
 For any $p\geq 1$, $t\geq 0$ be real numbers and $N\geq 2$, then 
 $$ \bigg(\exp(t)-\sum_{j=0}^{N-2}\frac{t^j}{j!}\bigg)^p\leq \exp(pt)-\sum_{j=0}^{N-2}\frac{(pt)^j}{j!}. $$
\end{corollary}
\begin{remark}\label{rem2.10}
 The function 
   $\Phi(t)=\exp(t)-\displaystyle\sum_{j=0}^{N-2}\frac{t^j}{j!}$ is increasing and convex in $[0,\infty)$.   
\end{remark}
The inequality in the following lemma is the singular Trudinger-Moser inequality, which Adimurthi and Yang first studied in \cite{MR2669653}.
\begin{lemma}\label{lem2.11}
For all $\alpha>0$, $0\leq \beta <N$ and $u\in W^{1,N}(\mathbb{R}^N)~(N\geq 2)$, then $\cfrac{\Phi\big(\alpha|u|^{N^\prime}\big)}{|x|^\beta}\in L^1(\mathbb{R}^N)$. Furthermore, we have for all $\alpha\leq\big(1-\frac{\beta}{N}\big)\alpha_N$ and $\tau>0$
$$ \sup_{\|u\|_{W^{1,N}_\tau}\leq 1} ~\int_{\mathbb{R}^N}\cfrac{\Phi\big(\alpha|u|^{N^\prime}\big)}{|x|^\beta}~\mathrm{d}x<+\infty,$$
where $\Phi$ is defined as in $(G1)$ and $\alpha_N=N\omega_{N-1}^\frac{1}{N-1}$, with $\omega_{N-1}$ is the volume of the unit sphere $S^{N-1}$. Also, the above inequality is sharp for $\alpha>\big(1-\frac{\beta}{N}\big)\alpha_N$, i.e., the supremum is infinity.    
\end{lemma}
\begin{remark}\label{rem2.12}
    Lemma \ref{lem2.11} is still true if we choose  $\tau=1$.
\end{remark}

Define the functional $\mathscr{L}: \mathbf{X}\to \mathbb{R}$ by
$$ \mathscr{L}(u)=\frac{1}{p}\int_{\mathbb{R}^N}(|\nabla u|^p+V(x)|u|^p)~\mathrm{d}x+\frac{1}{N}\int_{\mathbb{R}^N}(|\nabla u|^N+V(x)|u|^N)~\mathrm{d}x,~\text{for all}~u\in \mathbf{X}. $$
Then, $\mathscr{L}$ is well-defined and of class $\mathcal{C}^1(\mathbf{X},\mathbb{R})$. Let $\langle{\cdot,\cdot}\rangle$ is the duality pair between $\mathbf{X}^\ast$ and $\mathbf{X}$, then its Gateaux derivative is given by $$\langle{\mathscr{L}^\prime(u),v}\rangle= \big\langle{u,v}\big\rangle_{p,V}+\big\langle{u,v}\big\rangle_{N,V},\text{for all}~u,v\in \mathbf{X}.$$

The proof of the following lemma is quite similar to the proof of Lemma $6$ in \cite{MR4411714}; hence we omit it.
\begin{lemma}\label{lem2.12}
 Let $(V1)$ holds. Then the operator $\mathscr{L}^{\prime}: \mathbf{X}\to \mathbf{X}^\ast$ defined above has the following properties.
 \begin{itemize}
     \item [(a)] $\mathscr{L}^{\prime}$ is continuous, bounded and strictly monotone.
     \item [(b)] $\mathscr{L}^{\prime}$ is a map of type $(S)_{+}$, i.e., if $u_n\rightharpoonup u$ in $\mathbf{X}$ as $n\to\infty$ and  $\displaystyle\limsup_{n\to\infty}~\langle{\mathscr{L}^{\prime}(u_n)-\mathscr{L}^{\prime}(u),u_n-u}\rangle\leq 0$, then $u_n\to u~\text{in}~\mathbf{X}$ as $n\to\infty$. 
    \item [(c)]$\mathscr{L}^{\prime}$ is a homeomorphism.
 \end{itemize}
\end{lemma}
 To study the nonnegative solutions of \eqref{main problem}, define the energy functional $J_{\lambda_1,\lambda_2}:\mathbf{X}\to\mathbb{R}$ by
\begin{equation}\label{eq2.2}
  J_{\lambda_1,\lambda_2}(u)=\frac{1}{p} \widehat{M}\big( \|u\|_{W^{1,p}_V}^p\big)+\frac{1}{N} \widehat{M}\big( \|u\|_{W^{1,N}_V}^N\big)-\frac{\lambda_1}{q}\int_{\mathbb{R}^N}h(x)(u^{+})^q~\mathrm{d}x-\lambda_2\int_{\mathbb{R}^N}\cfrac{G(x,u)}{|x|^\beta}~\mathrm{d}x.  
\end{equation}
Under the assumption of $(G1)$ and the singular Trudinger-Moser inequality, one can easily verify that $J_{\lambda_1,\lambda_2}$ is well-defined, of class $\mathcal{C}^1(\mathbf{X},\mathbb{R})$ and its Gateaux derivative is defined by  
\begin{equation}\label{equation2.33} \langle{J^\prime_{\lambda_1,\lambda_2}(u),v}\rangle= M\big( \|u\|_{W^{1,p}_V}^p\big)\big\langle{u,v}\big\rangle_{p,V}+M\big( \|u\|_{W^{1,N}_V}^N\big)\big\langle{u,v}\big\rangle_{N,V}
     - \lambda_1\int_{\mathbb{R}^N}h(x)(u^{+})^{q-1}v~\mathrm{d}x-\lambda_2\int_{\mathbb{R}^N}\cfrac{g(x,u)}{|x|^\beta}v~\mathrm{d}x,
\end{equation}
for all $u,v\in\mathbf{X}$. The critical points of the energy functional $J_{\lambda_1,\lambda_2}$ are weak solutions of \eqref{main problem}.
\section{Proof of Theorem \ref{thm1.1}}\label{sec3}
In this section, for the sake of simplicity, the structural assumptions needed by Theorem \ref{thm1.1} hold, even if they are not explicitly stated. To prove Theorem \ref{thm1.1}, firstly, we shall prove that every nontrivial solution of \eqref{main problem} is nonnegative. Then, we established the mountain pass geometrical structures of the energy functional $J_{\lambda_1,\lambda_2}$.
\begin{lemma}\label{lem3.1}
 Any nontrivial critical point of the energy functional $J_{\lambda_1,\lambda_2}$ is nonnegative. 
\end{lemma}
\begin{proof}
    Let $u\in\mathbf{X}\setminus\{0\}$ be a critical point of the energy functional $J_{\lambda_1,\lambda_2}$. Write $u=u^+-u^-$ and choose $v=u^-$ as a test function in $\mathbf{X}$, then it follows  from \eqref{equation2.33} that
    $$ M\big( \|u\|_{W^{1,p}_V}^p\big)\big\|u^-\|_{W^{1,p}_V}^p+M\big( \|u\|_{W^{1,N}_V}^N\big)\big\|u^-\|_{W^{1,N}_V}^N=0.$$
    This together with $\|u\|_{W^{1,p}_V},\|u\|_{W^{1,N}_V}>0$ and $(M)$, yields
    $ \|u^-\|_{W^{1,p}_V}=\|u^-\|_{W^{1,N}_V}=0,~\text{that is},~\|u^-\|_{\mathbf{X}}=0.$
    Therefore, we get $u^-=0$ a.e. in $\mathbb{R}^N$ and hence $u=u^+\geq 0$ a.e. in $\mathbb{R}^N$. This completes the proof.
\end{proof}
\begin{lemma}\label{lem3.2}
 There exists a nonnegative function $e\in\mathbf{X}$, independent of $\lambda_1$, such that $\|e\|_{\mathbf{X}}\geq 2$ and  $J_{\lambda_1,\lambda_2}(e)<0$ hold true for all $\lambda_1,\lambda_2>0$.
\end{lemma}
\begin{proof}
    Choose $u\in\mathbf{X}\setminus\{0\}$, $u\geq 0$ with $\|u\|_{\mathbf{X}}=1$. Define the function $\psi:[1,\infty)\to\mathbb{R}$ by
    $$ \psi(t)=t^{-\mu}G(x,tu)-G(x,u),~\text{for all}~t\geq 1~\text{and}~u\in \mathbb{R}^+. $$
    By $(G2)$, we have 
    $ \psi^\prime(t)=t^{-\mu-1}\big( g(x,tu)tu-\mu G(x,tu)\big)\geq 0 ~\text{for all}~t\geq 1~\text{and}~u\in \mathbb{R}^+. $
It follows that $\psi(t)$ is an increasing function on $[1,\infty)$. Hence, we obtain $G(x,tu)\geq t^\mu G(x,t)~\text{for all}~t\geq 1~\text{and}~u\in \mathbb{R}^+.$ 
Choosing $t\geq2$ sufficiently large enough, we deduce from \eqref{eq2.2} that
 $$ J_{\lambda_1,\lambda_2}(tu)\leq \frac{2at^{N(k+1)}}{p(k+1)}-\lambda_2 t^\mu\int_{\mathbb{R}^N}\cfrac{G(x,u)}{|x|^\beta}~\mathrm{d}x\to -\infty~\text{as}~t\to\infty,~\text{since}~N(k+1)<\mu.$$ 
 The proof is completed by taking $e=tu$ with $t$ sufficiently large.
\end{proof}
\begin{lemma}\label{lem3.3}
 There exists $\rho\in(0,1]$ and two positive constants $\lambda_1^\ast$ and $\kappa$, depending upon $\rho$, such that $J_{\lambda_1,\lambda_2}(u)\geq \kappa$ for all $u\in \mathbf{X}$ with $\|u\|_{\mathbf{X}}=\rho$, and for all $\lambda_1\in(0,\lambda_1^\ast]$ and $\lambda_2>0$.
\end{lemma}
\begin{proof}
 By $(G3)$, for every $\epsilon>0$ and $\delta>0$ such that 
 \begin{equation}\label{eq3.1}
  G(x,t)\leq \frac{a}{N(k+1)}(\mathcal{S}_{N,\beta}-\epsilon)|t|^{N(k+1)},~\text{for any}~|t|\leq\delta,~x\in \mathbb{R}^N.   
 \end{equation}
 Furthermore, by $(G1),(G2)$ and remark \ref{rem2.10}, we have 
 \begin{equation}\label{eq3.2}
     G(x,t)\leq C_\delta |t|^{N(k+1)+1}\Phi(\alpha_0|t|^{N^\prime}),~\text{for any}~|t|\geq\delta,~x\in \mathbb{R}^N,
 \end{equation}
 where $$C_\delta=\frac{b_1}{\delta N(k+1)\Phi(\alpha_0\delta^{N^\prime})}+\frac{b_2}{\delta^{N(k+1)}} .$$
 Combining \eqref{eq3.1} and \eqref{eq3.2} together, we have for all $(x,t)\in \mathbb{R}^N\times \mathbb{R}$,
 \begin{equation}\label{eq3.3}
  G(x,t)\leq \frac{a}{N(k+1)}(\mathcal{S}_{N,\beta}-\epsilon)|t|^{N(k+1)}+ C_\delta |t|^{N(k+1)+1}\Phi(\alpha_0|t|^{N^\prime}).   
 \end{equation}
 Let $\delta\in(0,1]$ is sufficiently small enough and there holds $0<\|u\|_{\mathbf{X}}\leq\delta$. Choose $r>1$ satisfying $\frac{1}{r}+\frac{1}{r^\prime}=1$. Now, applying  H\"older's inequality, Lemma \ref{lem2.8} and Corollary \ref{cor2.9}, we have the following estimates:

\begin{align}\label{eq3.4}\int_{\mathbb{R}^N}\cfrac{|u|^{N(k+1)+1}\Phi\big(\alpha_0|u|^{N^\prime}\big)}{|x|^\beta}~\mathrm{d}x& \leq \Bigg\| \frac{|u|^{N(k+1)+1}}{|x|^{\frac{\beta}{r^\prime}}}\Bigg\|_{r^\prime}\Bigg\|\frac{\Phi\big(\alpha_0|u|^{N^\prime}\big)}{|x|^{\frac{\beta}{r}}}\Bigg\|_{r} \notag\\ 
&\leq \|u\|^{N(k+1)+1}_{L^{(N(k+1)+1)r^\prime}(\mathbb{R}^N,|x|^{-\beta})} \Bigg[\int_{\mathbb{R}^N}\cfrac{\Phi\big(r\alpha_0|u|^{N^\prime}\big)}{|x|^\beta}~\mathrm{d}x\Bigg]^{\frac{1}{r}}\notag\\
 &\leq \mathcal{S}^{-(N(k+1)+1)}_{(N(k+1)+1)r^\prime,\beta} \Bigg[\int_{\mathbb{R}^N}\cfrac{\Phi\big(r\alpha_0\|u\|_{\mathbf{X}}^{N^\prime}|\widetilde{u}|^{N^\prime}\big)}{|x|^\beta}~\mathrm{d}x\Bigg]^{\frac{1}{r}}\|u\|_{\mathbf{X}}^{N(k+1)+1},
\end{align}
 where $\widetilde{u}=u/\|u\|_{\mathbf{X}}$. Since $\|u\|_{\mathbf{X}}$ is very small, choose $r>1$ close to 1 and $\alpha>\alpha_0$ close to $\alpha_0$ such that $r\alpha\|u\|_{\mathbf{X}}^{N^\prime}\leq\big(1-\frac{\beta}{N}\big)\alpha_N$ holds. Take $\tau>0$ with $\tau\leq V_0$, then it is easy to see that $\|\widetilde{u}\|_{W^{1,N}_\tau}\leq \|\widetilde{u}\|_{\mathbf{X}}=1$. In view of remark \ref{rem2.10}, Lemma \ref{lem2.11} and \eqref{eq3.4}, there exists $C_1>0$ such that for all $u\in \mathbf{X}$ with $0<\|u\|_{\mathbf{X}}\leq\delta$, we have
\begin{equation}\label{eq3.5}
\int_{\mathbb{R}^N}\cfrac{|u|^{N(k+1)+1}\Phi\big(\alpha_0|u|^{N^\prime}\big)}{|x|^\beta}~\mathrm{d}x\leq C_1\mathcal{S}^{-(N(k+1)+1)}_{(N(k+1)+1)r^\prime,\beta}  \|u\|_{\mathbf{X}}^{N(k+1)+1}. 
\end{equation}
Hence, from \eqref{eq3.3} and \eqref{eq3.5}, we obtain for all $u\in \mathbf{X}$ with $\|u\|_{\mathbf{X}}\leq\delta$ sufficiently small enough,
\begin{align}\label{eq3.6} J_{\lambda_1,\lambda_2}(u)&\geq \frac{a}{N(k+1)}\big(\|u\|_{W^{1,p}_V}^{N(k+1)}+\|u\|_{W^{1,N}_V}^{N(k+1)}\big)-\frac{\lambda_1}{q}\mathcal{S}^{-q}_N\|h\|_\eta\|u\|^q_{\mathbf{X}}\notag\\
&\qquad-\frac{\lambda_2 a(\mathcal{S}_{N,\beta}-\epsilon)}{N(k+1)}\mathcal{S}^{-(N(k+1))}_{N(k+1),\beta}\|u\|^{N(k+1)}_{\mathbf{X}}-\lambda_2 C_1C_\delta\mathcal{S}^{-(N(k+1)+1)}_{(N(k+1)+1)r^\prime,\beta}  \|u\|_{\mathbf{X}}^{N(k+1)+1}\notag\\
     &\geq \frac{a(2^{1-N(k+1)} -\lambda_2C_\epsilon)}{N(k+1)} \|u\|^{N(k+1)}_{\mathbf{X}}-\frac{\lambda_1}{q}C_q\|h\|_\eta\|u\|^q_{\mathbf{X}}
    -C_{\lambda_2}\|u\|_{\mathbf{X}}^{N(k+1)+1},   
\end{align}
where $C_\epsilon=(\mathcal{S}_{N,\beta}-\epsilon)\mathcal{S}^{-(N(k+1))}_{N(k+1),\beta}$, $C_q=\mathcal{S}^{-q}_N$ and $C_{\lambda_2}=\lambda_2 C_1C_\delta\mathcal{S}^{-(N(k+1)+1)}_{(N(k+1)+1)r^\prime,\beta}$. Let us choose $C_\epsilon=2^{-N(k+1)}/\lambda_2$ and define the function
$$ f(\ell)=\frac{a}{2^{N(k+1)+1}N(k+1)}\ell^{N(k+1)}-C_{\lambda_2}\ell^{N(k+1)+1}, ~\text{where}~\ell\in[0,\delta]. $$
It is easy to see that $f$ admits a positive maximum $\kappa$ in $[0,\delta]$ at a point $\rho\in(0,\delta]$. Further, for all $u\in \mathbf{X}$ with $\|u\|_{\mathbf{X}}=\rho$, we have 
$$ J_{\lambda_1,\lambda_2}(u)\geq \frac{a}{2^{N(k+1)}N(k+1)}\rho^{N(k+1)} -\frac{\lambda_1}{q}C_q\|h\|_\eta\rho^q
    -C_{\lambda_2}\rho^{N(k+1)+1}\geq f(\rho)=\kappa>0,$$
for all $\lambda_1\in(0,\lambda_1^\ast]$, where $\lambda_1^\ast=\frac{aq}{2^{N(k+1)+1}N(k+1)C_q\|h\|_\eta}\rho^{N(k+1)-q}$. This completes the proof. 
\end{proof}
\begin{lemma}\label{lem3.4}
    Let $\rho\in(0,1]$ as in Lemma \ref{lem3.3}, then for all $\lambda_1\in(0,\lambda_1^\ast]$ as in Lemma \ref{lem3.3} and for all $\lambda_2>0$, there exist a sequence $\{u_n\}_n$ of nonnegative functions and some nonnegative function $u_{\lambda_1,\lambda_2}$ in $\overline{B}_\rho$ such that for all $n\in \mathbb{N}$,
    \begin{equation}\label{eq3.7}
    \begin{split}
     \|u_n\|_{\mathbf{X}}<\rho,~m_{\lambda_1,\lambda_2}\leq J_{\lambda_1,\lambda_2}(u_n)\leq m_{\lambda_1,\lambda_2}+\frac{1}{n},~u_n\rightharpoonup u_{\lambda_1,\lambda_2}~\text{in}~\mathbf{X},\\
     u_n\to u_{\lambda_1,\lambda_2}~\text{a.e. in}~\mathbb{R}^N~\text{and}~J^\prime_{\lambda_1,\lambda_2}(u_n)\to 0~\text{in}~\mathbf{X}^\ast~\text{as}~n\to\infty,\\
    \end{split}
\end{equation}
where $$ m_{\lambda_1,\lambda_2}=\inf\{J_{\lambda_1,\lambda_2}(u):u\in\overline{B}_\rho\}<0.$$    
\end{lemma}
\begin{proof}
Let $v\in\mathbf{X}\setminus\{0\}$, $v\geq 0$, with $\|v\|_{\mathbf{X}}=1$. By $(G2)$, we can deduce that $G(x,t)\geq C_2t^\mu$ for some constant $C_2>0$, for a.e. $x\in \mathbb{R}^N$ and for all $t>0$. Hence, for all $\ell\in (0,1]$ sufficiently small enough, we have 
\begin{align*}J_{\lambda_1,\lambda_2}(\ell v)&\leq \frac{a\ell^{p(k+1)}}{p(k+1)}\|v\|_{W^{1,p}_V}^{p(k+1)}+\frac{a\ell^{N(k+1)}}{N(k+1)}\|v\|_{W^{1,N}_V}^{N(k+1)}-\frac{\lambda_1\ell^q}{q}\|v\|^q_{q,h}-\lambda_2 C_2\ell^\mu\int_{\mathbb{R}^N}\frac{|v|^\mu}{|x|^\beta}~\mathrm{d}x\\
&\leq \frac{a\ell^{p(k+1)}}{p(k+1)}\|v\|_{\mathbf{X}}^{p(k+1)}+\frac{a\ell^{N(k+1)}}{N(k+1)}\|v\|_{\mathbf{X}}^{N(k+1)}-\frac{\lambda_1\ell^q}{q}\|v\|^q_{q,h}\\
&\leq \frac{2a\ell^{p(k+1)}}{p(k+1)}-\frac{\lambda_1\ell^q}{q}\|v\|^q_{q,h} <0,~\text{since}~q<p(k+1).\end{align*}
This together with \eqref{eq3.6}, yields 
$$-\infty<m_{\lambda_1,\lambda_2}=\inf _{\|u\|\leq \rho}J_{\lambda_1,\lambda_2}(u)\leq \inf _{\ell\in(0,\rho]}J_{\lambda_1,\lambda_2}(\ell v)<0. $$
It follows that 
\begin{equation}\label{eq3.8}
  -\infty< m_{\lambda_1,\lambda_2}=\inf\{J_{\lambda_1,\lambda_2}(u):u\in\overline{B}_\rho\}<0. 
\end{equation}
Thus, we infer that the functional $J_{\lambda_1,\lambda_2}$ is bounded from below and of class $\mathcal{C}^1$ on $\overline{B}_\rho$. Also, we know that $\overline{B}_\rho$ is a complete metric space with the metric given by the norm of $\mathbf{X}$. By Lemma \ref{lem3.3}, we also have 
\begin{equation}\label{eq3.9}
 \inf _{\partial B_\rho}J_{\lambda_1,\lambda_2}(u)\geq \kappa>0 . 
\end{equation}
In view of \eqref{eq3.8} and \eqref{eq3.9}, for $n$ large enough, we can choose
\begin{equation}\label{eq3.10}
    \frac{1}{n}\in \big(0, \inf _{\partial B_\rho}J_{\lambda_1,\lambda_2}(u)-\inf _{\overline{B}_\rho }J_{\lambda_1,\lambda_2}(u)\big).
\end{equation}
Applying Ekeland variational principle to the functional $J_{\lambda_1,\lambda_2}:\overline{B}_\rho\to\mathbb{R}$, we can find a sequence $\{u_n\}_n\subset \overline{B}_\rho $ such that
\begin{equation}\label{eq3.11}
  m_{\lambda_1,\lambda_2}\leq J_{\lambda_1,\lambda_2}(u_n)\leq m_{\lambda_1,\lambda_2}+\frac{1}{n}~\text{and} ~J_{\lambda_1,\lambda_2}(u_n)\leq J_{\lambda_1,\lambda_2}(u)+\frac{1}{n}\|u_n-u\|_{\mathbf{X}}, 
\end{equation}
for all $u\in\overline{B}_\rho$, with $u\neq u_n$ for each $n\in\mathbb{N}$. Now, it follows from \eqref{eq3.10} and \eqref{eq3.11} that
$$ J_{\lambda_1,\lambda_2}(u_n)\leq m_{\lambda_1,\lambda_2}+\frac{1}{n}=\inf _{\overline{B}_\rho }J_{\lambda_1,\lambda_2}(u)+\frac{1}{n}<\inf _{\partial B_\rho}J_{\lambda_1,\lambda_2}(u).$$
From this, we conclude that $\{u_n\}_n\subset B_\rho $, that is, $\|u_n\|_{\mathbf{X}}<\rho$, for all $n\in\mathbb{N}$. Let for all $\varphi\in B_1$, choose $t>0$ sufficiently small such that $u_n+t\varphi\in\overline{B}_\rho $ holds, for each $n\in\mathbb{N}$. Next, from \eqref{eq3.11}, we have
$$\langle{J^\prime_{\lambda_1,\lambda_2}(u_n),\varphi}\rangle=\lim_{t\to 0^+}\frac{J_{\lambda_1,\lambda_2}(u_n+t\varphi)-J_{\lambda_1,\lambda_2}(u_n)}{t}\geq -\frac{1}{n},~\text{for all}~\varphi\in B_1. $$
Due to the arbitrariness of $\varphi\in B_1$, we conclude that $|\langle{J^\prime_{\lambda_1,\lambda_2}(u_n),\varphi}\rangle|\leq \frac{1}{n}$, for all $\varphi\in B_1$. This together with \eqref{eq3.11}, we obtain 
\begin{equation}\label{eq3.12}
  J_{\lambda_1,\lambda_2}(u_n) \to m_{\lambda_1,\lambda_2}~\text{and}~ J^\prime_{\lambda_1,\lambda_2}(u_n)\to 0~\text{in}~\mathbf{X}^\ast~\text{as}~n\to\infty.
\end{equation}
Furthermore, since $\{u_n\}_n$ is bounded, there exists a subsequence still denoted by the same symbol and $u_{\lambda_1,\lambda_2},u_1,u_2\in \overline{B}_\rho$ such that $u_n\rightharpoonup u_{\lambda_1,\lambda_2},u^+_n\rightharpoonup u_1$ and $u^-_n\rightharpoonup u_2$ in $\mathbf{X}$ as $n\to\infty$. By Corollary \ref{cor2.6}, for any $\tau\in[p,p^\ast)\cup[N,\infty)$, we have $u_n\to u_{\lambda_1,\lambda_2},u^+_n\to u_1$ and $u^-_n\to u_2$ in $L^\tau(\mathbb{R}^N)$ as $n\to\infty$. Since the maps $u\mapsto u^{\pm}$ are continuous from $L^\tau(\mathbb{R}^N)$ into itself, therefore, we obtain $u_1=u^+_{\lambda_1,\lambda_2}$ and $u_2=u^-_{\lambda_1,\lambda_2}$. It follows that $u_n\to u_{\lambda_1,\lambda_2},u^+_n\to u^+_{\lambda_1,\lambda_2}$ and $u^-_n\to u^-_{\lambda_1,\lambda_2}$ a.e. in $\mathbb{R}^N$ as $n\to\infty$. In view of remark \ref{remark2.4}, up to a subsequence, we can assume that $\|u_n\|_{W^{1,p}_V}\to l_p$ and $\|u_n\|_{W^{1,N}_V}\to l_N$ as $n\to\infty$, where $l_p,l_N\geq 0$. Moreover, from \eqref{eq3.12}, we have $\langle{J^\prime_{\lambda_1,\lambda_2}(u_n),u^-_n}\rangle\to 0$ as $n\to\infty$. Hence, we obtain from \eqref{equation2.33} that
\begin{equation}\label{eq3.13}
  M(l_p^p)\big\|u_n^-\|_{W^{1,p}_V}^p+M( l_N^N)\big\|u_n^-\|_{W^{1,N}_V}^N=o_n(1)~\text{as} ~n\to\infty. 
\end{equation}
Now, we have the following possibilities based on $l_p$ and $l_N$.\\
\textit{Case-1}: Let $l_p=0$ and $l_N=0$.
This situation cannot occur. Indeed, if  $l_p=0$ and $l_N=0$, then $u_n\to 0$ in $\mathbf{X}$ as $n\to\infty$. Hence, we obtain from \eqref{eq3.12} that, $0=J_{\lambda_1,\lambda_2}(0) =m_{\lambda_1,\lambda_2}<0$, which is a contradiction.\\ 
\textit{Case-2}: Let either $l_p=0$ and $l_N\neq 0$ or $l_p\neq0$ and $l_N= 0$. We only study the situation when $l_p=0$ and $l_N\neq 0$. In this case, we obtain from \eqref{eq3.13} and $(M)$ that $u_n^-\to 0$ in $W^{1,N}_V(\mathbb{R}^N)$ as $n\to\infty$. Thus, by Lemma \ref{lem2.2}, we conclude that $u_n^-\to 0$ in $L^N(\mathbb{R}^N)$ as $n\to\infty$. It follows immediately that $u^-_{\lambda_1,\lambda_2}=0$ a.e. in $\mathbb{R}^N$ and hence $u_{\lambda_1,\lambda_2}=u^+_{\lambda_1,\lambda_2}\geq 0$ in $\mathbb{R}^N$. Consequently, we can assume that $u_n=u^+_n\geq 0$ in $\mathbb{R}^N$, since $u_n^-\to 0$ in $L^N(\mathbb{R}^N)$ as $n\to\infty$.\\  
\textit{Case-3}: Let $l_p\neq 0$ and $l_N\neq 0$.
In this situation, we deduce from \eqref{eq3.13} and $(M)$ that $u_n^-\to 0$ in $W^{1,p}_V(\mathbb{R}^N)$ and $W^{1,N}_V(\mathbb{R}^N)$ as $n\to\infty$. It follows from Lemma \ref{lem2.1} and Lemma \ref{lem2.2} that $u_n^-\to 0$ in $L^\tau(\mathbb{R}^N)$ as $n\to\infty$. Hence, we have $u^-_{\lambda_1,\lambda_2}=0$ a.e. in $\mathbb{R}^N$ and $u_{\lambda_1,\lambda_2}=u^+_{\lambda_1,\lambda_2}\geq 0$ in $\mathbb{R}^N$. Consequently, we can assume that $u_n=u^+_n\geq 0$ in $\mathbb{R}^N$, since $u_n^-\to 0$ in $L^\tau(\mathbb{R}^N)$ as $n\to\infty$.\\
The above three discussions conclude that the sequence $\{u_n\}_n$ and $u_{\lambda_1,\lambda_2}$ are nonnegative. This completes the proof.
\end{proof}
\begin{proof}[\textbf{Proof of Theorem \ref{thm1.1}}.]
Let us fix $\lambda_2>0$ and $\lambda_1\in (0,\widehat{\lambda}_1)$, with $\widehat{\lambda}_1=\text{min}\{\lambda_1^\ast,\lambda_1^0\}$, where $\lambda_1^\ast$ as in Lemma \ref{lem3.3}, while we set $\lambda_1^0$ as follows:
\begin{equation}\label{eq3.14}
 \lambda_1^0=\frac{aq\big(\text{min}\{1,V_0\}\big)^{k+1}(\mu-N(k+1))}{(\mu-q)N(k+1)\|h\|_{\eta}} \Bigg[\bigg(1-\frac{\beta}{N}\bigg)\frac{\alpha_N}{\alpha_0}\Bigg]^{\frac{N(k+1)-q}{N^\prime}} >0. 
\end{equation}
  By Lemma \ref{lem3.4}, there exists a sequence $\{u_n\}_n$ of nonnegative functions in $\overline{B}_\rho$ such that \eqref{eq3.7} holds. Consequently, from remark \ref{remark2.4}, Corollary \ref{cor2.6} and Lemma \ref{lem2.8}, we have
\begin{equation}\label{eq3.15}
    \begin{cases}
     u_n\rightharpoonup u_{\lambda_1,\lambda_2}~\text{in}~W^{1,p}_V(\mathbb{R}^N)~\text{and}~W^{1,N}_V(\mathbb{R}^N),\\
     u_n\to u_{\lambda_1,\lambda_2} ~\text{in}~L^\tau(\mathbb{R}^N),~\text{for any}~\tau\in[p,p^\ast)\cup[N,\infty),\\
     u_n\to u_{\lambda_1,\lambda_2} ~\text{in}~L^q(\mathbb{R}^N,|x|^{-\beta}),~\text{for any}~q\in [N,\infty),~\text{and}\\
      u_n\to u_{\lambda_1,\lambda_2} ~\text{a.e. in}~\mathbb{R}^N~\text{as}~n\to \infty.
    \end{cases}
\end{equation}
 Because of \eqref{eq3.15}, without loss of generality, up to a subsequence, we can assume that $\|u_n\|_{W^{1,p}_V}\to l_p$ and $\|u_n\|_{W^{1,N}_V}\to l_N$ as $n\to\infty$, where $l_p,l_N\geq 0$. Clearly, from \eqref{eq3.7}, we have $\langle{J^\prime_{\lambda_1,\lambda_2}(u_n),u_n-u_{\lambda_1,\lambda_2}}\rangle\to 0$ as $n\to\infty$. It follows that
\begin{equation}\label{eq3.16}
     \sum_{r\in\{p,N\}}\Big(M(l_r^r)\big\langle{u_n,u_n-u_{\lambda_1,\lambda_2}}\big\rangle_{r,V}\Big)- \lambda_1\int_{\mathbb{R}^N}h(x)u_n^{q-1}(u_n-u_{\lambda_1,\lambda_2})~\mathrm{d}x-\lambda_2\int_{\mathbb{R}^N}\cfrac{g(x,u_n)(u_n-u_{\lambda_1,\lambda_2})}{|x|^\beta}~\mathrm{d}x=o_n(1)
\end{equation}
 as $n\to\infty$. By Lemma \ref{lem3.4}, Lemma \ref{lem2.7} and $(G2)$, we obtain as $n\to\infty$
\begin{align*}0&>m_{\lambda_1,\lambda_2}= J_{\lambda_1,\lambda_2}(u_n)-\frac{1}{\mu}\langle{J^\prime_{\lambda_1,\lambda_2}(u_n),u_n}\rangle+o_n(1)\\
&\geq a\bigg(\frac{1}{p(k+1)}-\frac{1}{\mu}\bigg)\|u_n\|^{p(k+1)}_{W^{1,p}_V}+a\bigg(\frac{1}{N(k+1)}-\frac{1}{\mu}\bigg)\|u_n\|^{N(k+1)}_{W^{1,N}_V}-\lambda_1\bigg(\frac{1}{q}-\frac{1}{\mu}\bigg)\|h\|_\eta\|u_n\|^q_{W^{1,N}}+o_n(1)  \\
&\geq a\bigg(\frac{1}{N(k+1)}-\frac{1}{\mu}\bigg)\|u_n\|^{N(k+1)}_{W^{1,N}_V}-\lambda_1\bigg(\frac{1}{q}-\frac{1}{\mu}\bigg)\|h\|_\eta\|u_n\|^q_{W^{1,N}}+o_n(1). \end{align*}
This together with Lemma \ref{lem2.2}, we get as $n\to\infty$
$$ a\bigg(\frac{1}{N(k+1)}-\frac{1}{\mu}\bigg)\big(\text{min}\{1,V_0\}\big)^{k+1}\|u_n\|^{N(k+1)}_{W^{1,N}}-\lambda_1\bigg(\frac{1}{q}-\frac{1}{\mu}\bigg)\|h\|_\eta\|u_n\|^q_{W^{1,N}}+o_n(1) <0.$$
Consequently, for $\lambda_1\in (0,\widehat{\lambda}_1)$, with $\widehat{\lambda}_1=\text{min}\{\lambda_1^\ast,\lambda_1^0\}$, we have
\begin{equation}\label{eq3.17}
\limsup_{n\to\infty}\|u_n\|^{N^\prime}_{W^{1,N}}\leq\Bigg[\frac{\lambda_1(\mu-q)N(k+1)\|h\|_{\eta}}{aq\big(\text{min}\{1,V_0\}\big)^{k+1}(\mu-N(k+1))}\Bigg]^{\frac{N^\prime}{N(k+1)-q}}<\bigg(1-\frac{\beta}{N}\bigg)\frac{\alpha_N}{\alpha_0},  
\end{equation}
thanks to \eqref{eq3.14}. It follows from \eqref{eq3.17} that there exist $m_1>0$ and $n_0\in\mathbb{N}$ sufficiently large such that $\|u_n\|^{N^\prime}_{W^{1,N}}<m_1<\big(1-\frac{\beta}{N}\big)\frac{\alpha_N}{\alpha_0}$, for all $n\geq n_0$. Choose $r\geq N$, with $r^\prime=\frac{r}{r-1}>1$ satisfying $\frac{1}{r}+\frac{1}{r^\prime}=1$ . Let $r^\prime$ close to 1 and $\alpha>\alpha_0$ close to $\alpha_0$ such that we still have $r^\prime\alpha\|u_n\|^{N^\prime}_{W^{1,N}}<m_1<\big(1-\frac{\beta}{N}\big)\alpha_N$, for all $n\geq n_0$. Let us put
\begin{equation}\label{eq3.333}
 \widetilde{u}_n=\begin{cases}
        \frac{u_n}{\|u_n\|_{W^{1,N}}}~\text{if}~u_n\neq 0,\\
        0~~~~~~~~~~~~~\text{if}~u_n=0.
    \end{cases}   
\end{equation}
 Now, by $(G1)$, H\"older's inequality, Corollary \ref{cor2.9}, Lemma \ref{lem2.8} and \eqref{eq3.15}, we obtain
\begin{align*}\bigg|\int_{\mathbb{R}^N}\cfrac{g(x,u_n)(u_n-u_{\lambda_1,\lambda_2})}{|x|^\beta}~\mathrm{d}x\bigg|&\leq b_1\Bigg\| \frac{|u_n|^{N(k+1)-1}}{|x|^{\frac{\beta}{N^\prime}}}\Bigg\|_{N^\prime}\Bigg\| \frac{u_n-u_{\lambda_1,\lambda_2}}{|x|^{\frac{\beta}{N}}}\Bigg\|_{N}+b_2\Bigg\|\frac{\Phi\big(\alpha_0|u_n|^{N^\prime}\big)}{|x|^{\frac{\beta}{r^\prime}}}\Bigg\|_{r^\prime}\Bigg\| \frac{u_n-u_{\lambda_1,\lambda_2}}{|x|^{\frac{\beta}{r}}}\Bigg\|_{r} \\
&\leq b_1 \|u_n\|^{N(k+1)-1}_{L^{(N(k+1)-1)N^\prime}(\mathbb{R}^N,|x|^{-\beta})}\|u_n-u_{\lambda_1,\lambda_2}\|_{L^N(\mathbb{R}^N,|x|^{-\beta})}\\
&\qquad+b_2 \Bigg[\int_{\mathbb{R}^N}\cfrac{\Phi\big(r^\prime\alpha_0|u_n|^{N^\prime}\big)}{|x|^\beta}~\mathrm{d}x\Bigg]^{\frac{1}{r^\prime}}\|u_n-u_{\lambda_1,\lambda_2}\|_{L^r(\mathbb{R}^N,|x|^{-\beta})}\\
&\leq b_1\mathcal{S}^{-(N(k+1)-1)}_{(N(k+1)-1)N^\prime,\beta}\|u_n\|_{\mathbf{X}}^{N(k+1)-1} \|u_n-u_{\lambda_1,\lambda_2}\|_{L^N(\mathbb{R}^N,|x|^{-\beta})}\\
&\qquad+b_2\Bigg[\int_{\mathbb{R}^N}\cfrac{\Phi\big(r^\prime\alpha_0\|u_n\|_{{W^{1,N}}}^{N^\prime}|\widetilde{u}_n|^{N^\prime}\big)}{|x|^\beta}~\mathrm{d}x\Bigg]^{\frac{1}{r^\prime}}\|u_n-u_{\lambda_1,\lambda_2}\|_{L^r(\mathbb{R}^N,|x|^{-\beta})}\\
&\leq C_2\bigg(\|u_n-u_{\lambda_1,\lambda_2}\|_{L^N(\mathbb{R}^N,|x|^{-\beta})}+\|u_n-u_{\lambda_1,\lambda_2}\|_{L^r(\mathbb{R}^N,|x|^{-\beta})}\bigg) \to 0~\text{as}~n\to\infty,\end{align*}
where $$C_2=b_1\mathcal{S}^{-(N(k+1)-1)}_{(N(k+1)-1)N^\prime,\beta}\sup_{n\in\mathbb{N}}\|u_n\|_{\mathbf{X}}^{N(k+1)-1}+b_2\Bigg[\sup_{n\geq n_0}\int_{\mathbb{R}^N}\cfrac{\Phi\big(r^\prime\alpha_0\|u_n\|_{{W^{1,N}}}^{N^\prime}|\widetilde{u}_n|^{N^\prime}\big)}{|x|^\beta}~\mathrm{d}x\Bigg]^{\frac{1}{r^\prime}},$$
which is finite, thanks to Lemma \ref{lem2.11} and remark \ref{rem2.10}.
It follows that 
\begin{equation}\label{eq3.18}
  \lim_{n\to\infty}  \int_{\mathbb{R}^N}\cfrac{g(x,u_n)(u_n-u_{\lambda_1,\lambda_2})}{|x|^\beta}~\mathrm{d}x=0.
\end{equation}
 As $\frac{1}{\eta}+\frac{q-1}{N}+\frac{1}{N}=1$, therefore by  H\"older's inequality, Corollary \ref{cor2.6} and \eqref{eq3.15}, we have
 \begin{align*}\bigg|\int_{\mathbb{R}^N}h(x)u_n^{q-1}(u_n-u_{\lambda_1,\lambda_2})~\mathrm{d}x\bigg|&\leq \|h\|_\eta\|u_n\|^{q-1}_N\|u_n-u_{\lambda_1,\lambda_2}\|_N\\
 &\leq\mathcal{S}^{-(q-1)}_N \|h\|_\eta\|u_n\|^{q-1}_{\mathbf{X}}\|u_n-u_{\lambda_1,\lambda_2}\|_N\to 0~\text{as}~n\to\infty.\end{align*}
 Hence, we obtain
 \begin{equation}\label{eq3.19}
   \lim_{n\to\infty} \int_{\mathbb{R}^N}h(x)u_n^{q-1}(u_n-u_{\lambda_1,\lambda_2})~\mathrm{d}x=0. 
 \end{equation}
 Passing $n\to\infty$ in \eqref{eq3.16} and using \eqref{eq3.18} and \eqref{eq3.19}, we get
 \begin{equation}\label{eq3.20}
     M(l_p^p)\lim_{n\to\infty}\big\langle{u_n,u_n-u_{\lambda_1,\lambda_2}}\big\rangle_{p,V}+M( l_N^N)\lim_{n\to\infty}\big\langle{u_n,u_n-u_{\lambda_1,\lambda_2}}\big\rangle_{N,V}=0. 
 \end{equation}
 Due to the degenerate nature of the Kirchhoff function $M$, we have four cases, but we shall show that the first three cases cannot occur.\\
\textit{Case-1}: Let $l_p=0$ and $l_N=0$.
This situation cannot occur. Indeed, if  $l_p=0$ and $l_N=0$, then $u_n\to 0$ in $\mathbf{X}$ as $n\to\infty$. It follows from \eqref{eq3.12} that, $0=J_{\lambda_1,\lambda_2}(0) =m_{\lambda_1,\lambda_2}<0$, which is absurd situation.\\ 
\textit{Case-2}: Let $l_p\neq 0$ and $l_N=0$.
In this case, we obtain from \eqref{eq3.20} and $(M)$ that
\begin{equation}\label{eq3.21}
 \lim_{n\to\infty}\big\langle{u_n,u_n-u_{\lambda_1,\lambda_2}}\big\rangle_{p,V}=0.   
\end{equation}
Since $u_n\rightharpoonup u_{\lambda_1,\lambda_2}~\text{in}~W^{1,p}_V(\mathbb{R}^N)$ as $n\to\infty$, therefore
\begin{equation}\label{eq3.22}
 \lim_{n\to\infty}\big\langle{u_{\lambda_1,\lambda_2},u_n-u_{\lambda_1,\lambda_2}}\big\rangle_{p,V}=0.   
\end{equation}
Combining \eqref{eq3.21} and \eqref{eq3.22} together, we get
\begin{equation}\label{eq3.23}
    \begin{split}
     \lim_{n\to\infty}\int_{\mathbb{R}^N}\big(|\nabla u_n|^{p-2}\nabla u_n-|\nabla u_{\lambda_1,\lambda_2}|^{p-2}\nabla u_{\lambda_1,\lambda_2}\big).\big(\nabla u_n-\nabla u_{\lambda_1,\lambda_2}\big)~\mathrm{d}x\\
    +\lim_{n\to\infty}\int_{\mathbb{R}^N}V(x)\big(|u_n|^{p-2} u_n-| u_{\lambda_1,\lambda_2}|^{p-2} u_{\lambda_1,\lambda_2}\big)\big( u_n- u_{\lambda_1,\lambda_2}\big)~\mathrm{d}x=0.
    \end{split}
\end{equation}
By convexity and $(V1)$, we have
\begin{equation}\label{eq3.24}
    \begin{split}
     \big(|\nabla u_n|^{p-2}\nabla u_n-|\nabla u_{\lambda_1,\lambda_2}|^{p-2}\nabla u_{\lambda_1,\lambda_2}\big).\big(\nabla u_n-\nabla u_{\lambda_1,\lambda_2}\big)\geq 0~\text{a.e. in}~\mathbb{R}^N\\~\text{and}~
    V(x)\big(|u_n|^{p-2} u_n-| u_{\lambda_1,\lambda_2}|^{p-2} u_{\lambda_1,\lambda_2}\big)\big( u_n- u_{\lambda_1,\lambda_2}\big)\geq 0~\text{a.e. in}~\mathbb{R}^N
    \end{split}
\end{equation}
for any $n\in\mathbb{N}$. Now, from \eqref{eq3.23} and \eqref{eq3.24}, we obtain
\begin{equation}\label{eq3.25}
  \lim_{n\to\infty}\int_{\mathbb{R}^N}\big(|\nabla u_n|^{p-2}\nabla u_n-|\nabla u_{\lambda_1,\lambda_2}|^{p-2}\nabla u_{\lambda_1,\lambda_2}\big).\big(\nabla u_n-\nabla u_{\lambda_1,\lambda_2}\big)~\mathrm{d}x=0
\end{equation}
and
\begin{equation}\label{eq3.26}
   \lim_{n\to\infty}\int_{\mathbb{R}^N}V(x)\big(|u_n|^{p-2} u_n-| u_{\lambda_1,\lambda_2}|^{p-2} u_{\lambda_1,\lambda_2}\big)\big( u_n- u_{\lambda_1,\lambda_2}\big)~\mathrm{d}x=0.
\end{equation}
Using Simon's inequality (see Lemma 4.2 of \cite{MR3732173}), for $p\geq 2$, there exists $c_p>0$ such that
\begin{align*} \int_{\mathbb{R}^N}|\nabla u_n-\nabla u_{\lambda_1,\lambda_2}|^p~\mathrm{d}x &\leq c_p\int_{\mathbb{R}^N}\big(|\nabla u_n|^{p-2}\nabla u_n-|\nabla u_{\lambda_1,\lambda_2}|^{p-2}\nabla u_{\lambda_1,\lambda_2}\big).\big(\nabla u_n-\nabla u_{\lambda_1,\lambda_2}\big)~\mathrm{d}x\\
&\quad\to 0~\text{as}~n\to\infty,~(\text{thanks to}~\eqref{eq3.25}).\end{align*}
Finally, assume $1<p< 2$, then by Simon's inequality \& H\"older's inequality, there exists $C_p>0$ such that
\begin{align*} \int_{\mathbb{R}^N}|\nabla u_n-\nabla u_{\lambda_1,\lambda_2}|^p~\mathrm{d}x &\leq C_p\int_{\mathbb{R}^N}\Big[\big(|\nabla u_n|^{p-2}\nabla u_n-|\nabla u_{\lambda_1,\lambda_2}|^{p-2}\nabla u_{\lambda_1,\lambda_2}\big).\big(\nabla u_n-\nabla u_{\lambda_1,\lambda_2}\big)\Big]^{\frac{p}{2}}\\
&\qquad\times\Big[|\nabla u_n|^{p}+|\nabla u_{\lambda_1,\lambda_2}|^{p}\Big]^{\frac{2-p}{2}}~\mathrm{d}x\\
&\leq C_p\int_{\mathbb{R}^N}\Big[\big(|\nabla u_n|^{p-2}\nabla u_n-|\nabla u_{\lambda_1,\lambda_2}|^{p-2}\nabla u_{\lambda_1,\lambda_2}\big).\big(\nabla u_n-\nabla u_{\lambda_1,\lambda_2}\big)\Big]^{\frac{p}{2}} \\
&\qquad\times\Big[\big(|\nabla u_n|^{p}\big)^{\frac{2-p}{2}}+\big(|\nabla u_{\lambda_1,\lambda_2}|^{p}\big)^{\frac{2-p}{2}}\Big]~\mathrm{d}x\\
&\leq C_p\Bigg[\int_{\mathbb{R}^N}\big(|\nabla u_n|^{p-2}\nabla u_n-|\nabla u_{\lambda_1,\lambda_2}|^{p-2}\nabla u_{\lambda_1,\lambda_2}\big).\big(\nabla u_n-\nabla u_{\lambda_1,\lambda_2}\big)~\mathrm{d}x\Bigg]^{\frac{p}{2}} \\
&\qquad\times\Big[\|\nabla u_n\|_p^{\frac{(2-p)p}{2}}+\|\nabla u_{\lambda_1,\lambda_2}\|_p^{\frac{(2-p)p}{2}}\Big]\to 0~\text{as}~n\to\infty, \end{align*}
thanks to \eqref{eq3.25} and the fact that $\big\{\|\nabla u_n\|_p^{\frac{(2-p)p}{2}}+\|\nabla u_{\lambda_1,\lambda_2}\|_p^{\frac{(2-p)p}{2}}\big\}_n$ is bounded in $\mathbb{R}$. Hence, we conclude that  $\nabla u_n\to\nabla u_{\lambda_1,\lambda_2}$ in $L^p(\mathbb{R}^N)$ as $n\to\infty$, for any $p> 1$. Arguing similarly and using \eqref{eq3.26}, we can easily deduce that $u_n\to u_{\lambda_1,\lambda_2}$ in $L^p_V(\mathbb{R}^N)$ as $n\to\infty$, for any $p> 1$. Consequently, $\nabla u_n\to\nabla u_{\lambda_1,\lambda_2}$ and $u_n\to u_{\lambda_1,\lambda_2}$ a.e. in $\mathbb{R}^N$ as $n\to\infty$. Notice that $u_n\to u_{\lambda_1,\lambda_2}$ in $W^{1,p}_V(\mathbb{R}^N)$ as $n\to\infty$ and hence we have $\|u_n\|_{W^{1,p}_V}\to \|u_{\lambda_1,\lambda_2}\|_{W^{1,p}_V} $ in $\mathbb{R}$ as $n\to\infty$. Therefore, we obtain
\begin{equation}\label{eq3.27}
 \|u_{\lambda_1,\lambda_2}\|_{W^{1,p}_V}=l_p>0.   
\end{equation}
Since $l_N=0$, therefore $\|u_n\|_{W^{1,N}_V}\to 0 $ in $\mathbb{R}$ as $n\to\infty$. It follows that $\nabla u_n\to 0$ in $L^N(\mathbb{R}^N)$ and $ u_n\to 0$ in $L^N_V(\mathbb{R}^N)$ as $n\to\infty$. Hence, we obtain $u_{\lambda_1,\lambda_2}=0$ a.e. in $\mathbb{R}^N$ and $\nabla u_{\lambda_1,\lambda_2}=0$ a.e. in $\mathbb{R}^N$, which contradicts \eqref{eq3.27}. So, this situation cannot occur.\\
\textit{Case-3}: Let $l_p= 0$ and $l_N\neq 0$.
The proof of this case is similar to \textit{Case-2}. Moreover, this situation is also impossible.\\ 
\textit{Case-4}: Let $l_p\neq 0$ and $l_N\neq 0$.
Since $u_n\rightharpoonup u_{\lambda_1,\lambda_2}~\text{in}~W^{1,r}_V(\mathbb{R}^N)$ as $n\to\infty$, for $r\in\{p,N\}$, therefore from \eqref{eq3.20}, we have
\begin{equation}\label{eq3.28}
    \begin{split}
     \sum_{r\in\{p,N\}}\Bigg(M(l_r^r)\lim_{n\to\infty}\Bigg[\int_{\mathbb{R}^N}\big(|\nabla u_n|^{r-2}\nabla u_n-|\nabla u_{\lambda_1,\lambda_2}|^{r-2}\nabla u_{\lambda_1,\lambda_2}\big).\big(\nabla u_n-\nabla u_{\lambda_1,\lambda_2}\big)~\mathrm{d}x\\
    +\int_{\mathbb{R}^N}V(x)\big(|u_n|^{r-2} u_n-| u_{\lambda_1,\lambda_2}|^{r-2} u_{\lambda_1,\lambda_2}\big)\big( u_n- u_{\lambda_1,\lambda_2}\big)~\mathrm{d}x\Bigg]\Bigg)=0.
    \end{split}
\end{equation}
Due to convexity of the map $t\mapsto t^r$ for $r\in\{p,N\}$, $ (V1)$ and the fact that $M(l_p^p),M(l_N^N)>0$, we obtain from \eqref{eq3.28} that
$$\text{min} \big\{M(l_p^p),M(l_N^N)\big\}\limsup_{n\to\infty}\langle{\mathscr{L}^{\prime}(u_n)-\mathscr{L}^{\prime}(u_{\lambda_1,\lambda_2}),u_n-u_{\lambda_1,\lambda_2}}\rangle\leq 0.$$
It follows that $$\limsup_{n\to\infty}\langle{\mathscr{L}^{\prime}(u_n)-\mathscr{L}^{\prime}(u_{\lambda_1,\lambda_2}),u_n-u_{\lambda_1,\lambda_2}}\rangle\leq 0. $$
This together with Lemma \ref{lem2.12}, we deduce that $u_n\to u_{\lambda_1,\lambda_2}$ in $\mathbf{X}$ as $n\to\infty$. Now using the fact that $J_{\lambda_1,\lambda_2}\in\mathcal{C}^1(\mathbf{X},\mathbb{R})$, we get $J_{\lambda_1,\lambda_2}(u_{\lambda_1,\lambda_2})=m_{\lambda_1,\lambda_2}<0$ and $J^\prime_{\lambda_1,\lambda_2}(u_{\lambda_1,\lambda_2})=0$. This implies at once $u_{\lambda_1,\lambda_2}\neq 0$ and $u_{\lambda_1,\lambda_2}$ is a nontrivial nonnegative solution of \eqref{main problem}.

Now, since the nonnegative solution $u_{\lambda_1,\lambda_2}\in \overline{B}_\rho $, with $\rho$ independent of $\lambda_1$ as stated in Lemma \ref{lem3.4}, therefore $\{u_{\lambda_1,\lambda_2} \}_{\lambda_1\in (0,\widehat{\lambda}_1)}$ is uniformly bounded in $\mathbf{X}$. Moreover, by $(G2)$, Lemma \ref{lem2.7} and the fact that $u_n\to u_{\lambda_1,\lambda_2}$ in $\mathbf{X}$ as $n\to\infty$ , it is easy to see that 
\begin{align*}0&>m_{\lambda_1,\lambda_2}=\lim_{n\to\infty}\Big[ J_{\lambda_1,\lambda_2}(u_n)-\frac{1}{\mu}\langle{J^\prime_{\lambda_1,\lambda_2}(u_n),u_n}\rangle\Big]\\
&\geq a\bigg(\frac{1}{p(k+1)}-\frac{1}{\mu}\bigg)\|u_{\lambda_1,\lambda_2}\|^{p(k+1)}_{W^{1,p}_V}+a\bigg(\frac{1}{N(k+1)}-\frac{1}{\mu}\bigg)\|u_{\lambda_1,\lambda_2}\|^{N(k+1)}_{W^{1,N}_V}-\lambda_1\bigg(\frac{1}{q}-\frac{1}{\mu}\bigg)\|u_{\lambda_1,\lambda_2}\|^q_{q,h} \\
&\geq a\bigg(\frac{1}{p(k+1)}-\frac{1}{\mu}\bigg)\|u_{\lambda_1,\lambda_2}\|^{p(k+1)}_{W^{1,p}_V}+a\bigg(\frac{1}{N(k+1)}-\frac{1}{\mu}\bigg)\|u_{\lambda_1,\lambda_2}\|^{N(k+1)}_{W^{1,N}_V}-\lambda_1 D_{q,h} ,\end{align*}
where $$D_{q,h}=\bigg(\frac{1}{q}-\frac{1}{\mu}\bigg)\mathcal{S}^{-q}_{q,h}\sup_{\lambda_1\in (0,\widehat{\lambda}_1)}\|u_{\lambda_1,\lambda_2}\|^q_{\mathbf{X}}<\infty.$$
Therefore, we have
$$0\geq\limsup_{\lambda_1\to 0^+} m_{\lambda_1,\lambda_2}\geq \limsup_{\lambda_1\to 0^+}\Bigg[a\bigg(\frac{1}{p(k+1)}-\frac{1}{\mu}\bigg)\|u_{\lambda_1,\lambda_2}\|^{p(k+1)}_{W^{1,p}_V} +a\bigg(\frac{1}{N(k+1)}-\frac{1}{\mu}\bigg)\|u_{\lambda_1,\lambda_2}\|^{N(k+1)}_{W^{1,N}_V}\Bigg]\geq 0.$$
This shows that $$\lim_{\lambda_1\to 0^+}\|u_{\lambda_1,\lambda_2}\|_{W^{1,p}_V}=\lim_{\lambda_1\to 0^+}\|u_{\lambda_1,\lambda_2}\|_{W^{1,N}_V}=0,~\text{i.e.,}~\lim_{\lambda_1\to 0^+}\|u_{\lambda_1,\lambda_2}\|_{\mathbf{X}}=0.$$
Hence the proof of Theorem \ref{thm1.1} is completed.
\end{proof}
\section{Proof of Theorem \ref{thm1.2}}\label{sec4}
In this section, for the sake of simplicity, we assume that the structural assumptions needed by Theorem \ref{thm1.2} hold. By Lemma \ref{lem3.2} \& Lemma \ref{lem3.3}, we see that the energy functional $J_{\lambda_1,\lambda_2}$ satisfies the mountain pass geometrical structures. To apply the mountain pass theorem, we must check the validity of the Palais-Smale compactness condition at a suitable level $c$. We say $\{u_n\}_n\subset\mathbf{X}$ is a $(PS)_c$ sequence for $J_{\lambda_1,\lambda_2}$ at any suitable level $c$ if 
\begin{equation}\label{eq4.1}
  J_{\lambda_1,\lambda_2}(u_n)\to c~\text{and}~\sup_{\|\phi\|_{\mathbf{X}}=1}|\langle{J^\prime_{\lambda_1,\lambda_2}(u_n),\phi}\rangle|\to 0~\text{as}~n\to\infty.  
\end{equation}
Note that if this sequence has a convergent subsequence in $\mathbf{X}$, then we say $J_{\lambda_1,\lambda_2}$ satisfies $(PS)_c$ condition at any suitable level $c$. 
\begin{lemma}\label{lem4.1}
 Let $\lambda_1>0$ and $\lambda_2>0$ be fixed. Then, the $(PS)_c$ sequence $\{u_n\}_n\subset\mathbf{X}$ for $J_{\lambda_1,\lambda_2}$ at any level $c>0$ is bounded in $\mathbf{X}$ and satisfying
 $$ c+\Gamma\lambda_1^{\sigma}+o_n(1)\geq \Upsilon\Big(\|u_n\|^{p(k+1)}_{W^{1,p}_V}+\|u_n\|^{N(k+1)}_{W^{1,N}_V}\Big)~\text{as}~n\to\infty, $$
 where  $\sigma$, $\Gamma$ and $\Upsilon$ are defined by
 $$\sigma=\frac{N(k+1)}{N(k+1)-q},~\Gamma=\frac{(\mu-q)S^{-q}_N \|h\|^\sigma_\eta}{q\mu\sigma}\bigg(\frac{2(\mu-q)S^{-q}_N}{a(\mu-N(k+1))}\bigg)^\frac{q\sigma}{N(k+1)},~\text{and}~\Upsilon=\frac{a(\mu-N(k+1))}{2\mu N(k+1)}, $$
 where $S_N$ is the best constant in the embedding $W^{1,N}_V(\mathbb{R}^N)\hookrightarrow L^{N}(\mathbb{R}^N)$.
\end{lemma}
\begin{proof}
Let $\{u_n\}_n\subset\mathbf{X}$ be a $(PS)_c$ sequence for $J_{\lambda_1,\lambda_2}$ at level $c>0$. Thus, \eqref{eq4.1} is satisfied.
Hence, we obtain  
$$ \bigg\langle{J^\prime_{\lambda_1,\lambda_2}(u_n),\frac{u_n}{\|u_n\|_{\mathbf{X}}}}\bigg\rangle=o_n(1)~\text{and}~J_{\lambda_1,\lambda_2}(u_n)=c+o_n(1)~\text{as}~n\to\infty.$$
This together with $(G2)$, Lemma \ref{lem2.2} and Young's inequality with $\epsilon$, we have as $n\to\infty$
\begin{align*}c+o_n(1)+o_n(1)\|u_n\|_{\mathbf{X}}&=J_{\lambda_1,\lambda_2}(u_n)-\frac{1}{\mu}\langle{J^\prime_{\lambda_1,\lambda_2}(u_n),u_n}\rangle \\
&\geq a\bigg(\frac{1}{p(k+1)}-\frac{1}{\mu}\bigg)\|u_n\|^{p(k+1)}_{W^{1,p}_V}+a\bigg(\frac{1}{N(k+1)}-\frac{1}{\mu}\bigg)\|u_n\|^{N(k+1)}_{W^{1,N}_V}\\ &\qquad-\lambda_1 S^{-q}_N\bigg(\frac{1}{q}-\frac{1}{\mu}\bigg)\|h\|_\eta\|u_n\|^q_{W^{1,N}_V} \\
&\geq a\bigg(\frac{1}{p(k+1)}-\frac{1}{\mu}\bigg)\|u_n\|^{p(k+1)}_{W^{1,p}_V}+\bigg[a\bigg(\frac{1}{N(k+1)}-\frac{1}{\mu}\bigg)-\epsilon S^{-q}_N\bigg(\frac{1}{q}-\frac{1}{\mu}\bigg)\bigg]\|u_n\|^{N(k+1)}_{W^{1,N}_V}\\
&\qquad-C_\epsilon S^{-q}_N\bigg(\frac{1}{q}-\frac{1}{\mu}\bigg)\lambda_1^{\sigma}\|h\|_\eta^{\sigma}.\end{align*}
Set $\epsilon=a\big(\frac{1}{N(k+1)}-\frac{1}{\mu}\big)\big/2S^{-q}_N\big(\frac{1}{q}-\frac{1}{\mu}\big)$, $\Gamma=C_\epsilon S^{-q}_N\big(\frac{1}{q}-\frac{1}{\mu}\big)\|h\|_\eta^{\sigma}$ and $\Upsilon=\frac{a}{2}\big(\frac{1}{N(k+1)}-\frac{1}{\mu}\big)$, then we get from the above inequality that
\begin{equation}\label{eq4.2}
 c+\Gamma\lambda_1^{\sigma}+o_n(1)+o_n(1)\|u_n\|_{\mathbf{X}}\geq \Upsilon\Big(\|u_n\|^{p(k+1)}_{W^{1,p}_V}+\|u_n\|^{N(k+1)}_{W^{1,N}_V}\Big)~\text{as}~n\to\infty.
 \end{equation}
If possible, let $\{u_n\}_n$ be not bounded in $\mathbf{X}$, then we can divide into three cases to prove the conclusion of this Lemma.\\
\textit{Case-1:}
Let $\|u_n\|_{W^{1,p}_V}\to\infty$ and $\|u_n\|_{W^{1,N}_V}\to\infty$ as $n\to\infty$. Since $p(k+1)<N(k+1)$, therefore we have $\|u_n\|^{N(k+1)}_{W^{1,N}_V}\geq \|u_n\|^{p(k+1)}_{W^{1,N}_V}>1$ for $n$ large enough. It follows from \eqref{eq4.2} that
$$ c+\Gamma\lambda_1^{\sigma}+o_n(1)+o_n(1)\|u_n\|_{\mathbf{X}}\geq \Upsilon\Big(\|u_n\|^{p(k+1)}_{W^{1,p}_V}+\|u_n\|^{p(k+1)}_{W^{1,N}_V}\Big)\geq  2^{1-p(k+1)}\Upsilon\|u_n\|^{p(k+1)}_{\mathbf{X}}~\text{as}~n\to\infty.$$
Dividing $\|u_n\|^{p(k+1)}_{\mathbf{X}}$ on both the sides and letting $n\to\infty$, we get $0\geq 2^{1-p(k+1)}\Upsilon>0 $, which is a contradiction.\\
\textit{Case-2:}
Let $\|u_n\|_{W^{1,p}_V}\to\infty$ as $n\to\infty$ and $\|u_n\|_{W^{1,N}_V}$ is bounded. From \eqref{eq4.2}, we get
$$ c+\Gamma\lambda_1^{\sigma}+o_n(1)+o_n(1)\|u_n\|_{\mathbf{X}}\geq \Upsilon\|u_n\|^{p(k+1)}_{W^{1,p}_V}~\text{as}~n\to\infty.$$
Dividing $\|u_n\|^{p(k+1)}_{W^{1,p}_V}$ on both the sides and letting $n\to\infty$, we get $0\geq \Upsilon>0 $, which is again a contradiction.\\
\textit{Case-3:}
Let $\|u_n\|_{W^{1,N}_V}\to\infty$ as $n\to\infty$ and $\|u_n\|_{W^{1,p}_V}$ is bounded. Similar to \textit{Case-2}, we get a contradiction.\\
Consequently, we infer from the above three discussions that $\{u_n\}_n\subset\mathbf{X}$ must be bounded. Therefore, based on \eqref{eq4.2}, the proof of the lemma is completed.
\end{proof}
\begin{lemma}\label{lem4.2}
Suppose $\lambda_1\in(0,\lambda_1^\ast]$, where $\lambda_1^\ast$ as in Lemma \ref{lem3.3} and $\lambda_2>0$ be fixed. Assume $c<c_0-\Gamma\lambda_1^{\sigma}$, where $c_0>0$ such that 
\begin{equation}\label{eq4.3}
  c_0\leq \Upsilon\big(\text{min}\{1,V_0\}\big)^{k+1}\Bigg[\bigg(1-\frac{\beta}{N}\bigg)\frac{\alpha_N}{\alpha_0}\Bigg]^{\frac{N(k+1)}{N^\prime}},   
\end{equation}
where $\sigma$, $\Gamma$ and $\Upsilon$ are as in Lemma \ref{lem4.1}. Then, for any $(PS)_c$ sequence $\{u_n\}_n\subset\mathbf{X}$ of $J_{\lambda_1,\lambda_2}$ at level $c>0$, the following holds
\begin{equation}\label{eq4.4}
 \limsup_{n\to\infty}\|u_n\|^{N^\prime}_{W^{1,N}}<\bigg(1-\frac{\beta}{N}\bigg)\frac{\alpha_N}{\alpha_0}.   
\end{equation}
Consequently, if $\alpha>\alpha_0$, where $\alpha_0$ as in $(G4)$, then there exists $n_0\in\mathbb{N}$ sufficiently large enough such that
\begin{align*}
 \sup_{n\geq n_0}\int_{\mathbb{R}^N}\cfrac{|g(x,u_n)u_n|}{|x|^\beta}~\mathrm{d}x<\infty.    
\end{align*}
\end{lemma}
\begin{proof}
    Due to Lemma \ref{lem3.2} and Lemma \ref{lem3.3}, we conclude that there exists a $(PS)_c$ sequence $\{u_n\}_n\subset\mathbf{X}$ for $J_{\lambda_1,\lambda_2}$ at level $c>0$ satisfying $c<c_0-\Gamma\lambda_1^{\sigma}$. Consequently, by Lemma \ref{lem4.1} and Lemma \ref{lem2.2}, we have
    \begin{equation}\label{eq4.6}
      \Upsilon\big(\text{min}\{1,V_0\}\big)^{k+1}\|u_n\|^{N(k+1)}_{W^{1,N}}\leq\Upsilon\|u_n\|^{N(k+1)}_{W^{1,N}_V}\leq  c+\Gamma\lambda_1^{\sigma}+o_n(1)~\text{as}~n\to\infty.
    \end{equation}
    In view of \eqref{eq4.3} and \eqref{eq4.6}, we deduce that \eqref{eq4.4} holds. Now by $(G4)$ and remark \ref{rem2.10}, we have for all $(x,u)\in \mathbb{R}^N\times \mathbb{R}$,
    \begin{align*}\big|g(x,u)\big|&= \bigg|\int_{0}^{1}\frac{\mathrm{d}}{\mathrm{d}s}g(x,s u)~\mathrm{d}s\bigg|= \bigg|\int_{0}^{1}\frac{1}{s}\partial_u g(x,s u)s u~\mathrm{d}s\bigg|\leq\int_{0}^{1}\frac{1}{s}|\partial_u g(x,s u)s u|~\mathrm{d}s \\
    &\leq b_1|u|^{N(k+1)-1}\int_{0}^{1}s^{N(k+1)-2}~\mathrm{d}s+b_2 \int_{0}^{1}\sum_{j=N-1}^{\infty}\frac{\alpha_0^j|u|^{jN^\prime}}{j!}s^{jN^\prime-1}~\mathrm{d}s\end{align*}\begin{align*}
    \qquad\leq& b_1|u|^{N(k+1)-1}+b_2 \sum_{j=N-1}^{\infty}\frac{\alpha_0^j|u|^{jN^\prime}}{j!}= b_1 |u|^{N(k+1)-1}+b_2 \Phi(\alpha_0 |u|^{N^\prime}). \end{align*}
    It follows that for all $(x,u)\in \mathbb{R}^N\times \mathbb{R}$, we have
    \begin{equation}\label{eq4.7}
      \big|g(x,u)u\big|\leq b_1 |u|^{N(k+1)}+b_2|u| \Phi(\alpha_0 |u|^{N^\prime}). 
    \end{equation}
Now from \eqref{eq4.4}, there exist $\delta>0$ and $n_0\in\mathbb{N}$ sufficiently large such that $\|u_n\|^{N^\prime}_{W^{1,N}}<\delta<\big(1-\frac{\beta}{N}\big)\frac{\alpha_N}{\alpha_0}$, for all $n\geq n_0$. Choose $r\geq N$, with $r^\prime=\frac{r}{r-1}>1$ satisfying $\frac{1}{r}+\frac{1}{r^\prime}=1$ . Let $r^\prime$ close to 1 and $\alpha>\alpha_0$ close to $\alpha_0$ such that we still have $r^\prime\alpha\|u_n\|^{N^\prime}_{W^{1,N}}<\delta<\big(1-\frac{\beta}{N}\big)\alpha_N$, for all $n\geq n_0$. Let us suppose \eqref{eq3.333} is satisfied. Hence, from \eqref{eq4.7}, H\"older's inequality, Lemma \ref{lem2.8}, Corollary \ref{cor2.9}, remark \ref{rem2.10} and  Lemma \ref{lem2.11}, we have for all $n\geq n_0$
\begin{align*}\int_{\mathbb{R}^N}\cfrac{|g(x,u_n)u_n|}{|x|^\beta}~\mathrm{d}x&\leq b_1\|u_n\|^{N(k+1)}_{L^{N(k+1)}(\mathbb{R}^N,|x|^{-\beta})}+b_2\Bigg\|\frac{\Phi\big(\alpha_0|u_n|^{N^\prime}\big)}{|x|^{\frac{\beta}{r^\prime}}}\Bigg\|_{r^\prime}\| {u_n}\|_{L^{r}(\mathbb{R}^N,|x|^{-\beta})} \\
&\leq b_1\mathcal{S}^{-(N(k+1))}_{N(k+1),\beta}\|u_n\|^{N(k+1)}_{\mathbf{X}}+b_2\mathcal{S}^{-1}_{r,\beta}\Bigg[\int_{\mathbb{R}^N}\cfrac{\Phi\big(r^\prime\alpha_0|u_n|^{N^\prime}\big)}{|x|^\beta}~\mathrm{d}x\Bigg]^{\frac{1}{r^\prime}}\|u_n\|_{\mathbf{X}} \\
&\leq b_1\mathcal{S}^{-(N(k+1))}_{N(k+1),\beta} \|u_n\|^{N(k+1)}_{\mathbf{X}}+b_2\mathcal{S}^{-1}_{r,\beta}\Bigg[\sup_{n\geq n_0}\int_{\mathbb{R}^N}\cfrac{\Phi\big(r^\prime\alpha_0\|u_n\|^{N^\prime}_{W^{1,N}}|\widetilde{u}_n|^{N^\prime}\big)}{|x|^\beta}~\mathrm{d}x\Bigg]^{\frac{1}{r^\prime}}\|u_n\|_{\mathbf{X}}<\infty,\end{align*} since $\{u_n\}_n$ is bounded in $\mathbf{X}$, thanks to Lemma \ref{lem4.1}. Hence, the proof is completed.
\end{proof}
The next lemma is devoted to singular exponential nonlinearity in the spirit of Brézis and Lieb. 
\begin{lemma}\label{lem4.3}
 Suppose $\{u_n\}_n$ be a sequence in $\mathbf{X}$ and $u$ be in $\mathbf{X}$ such that $u_n\rightharpoonup u$ in $\mathbf{X}$, $u_n\to u$ a.e. in $\mathbb{R}^N$ and $\nabla u_n\to \nabla u$ a.e. in $\mathbb{R}^N$ as $n\to\infty$. If
 \begin{equation}\label{eq4.8}
 \limsup_{n\to\infty}\|u_n\|^{N^\prime}_{W^{1,N}}<\bigg(1-\frac{\beta}{N}\bigg)\frac{\alpha_N}{2^{N^\prime}\alpha_0}   
\end{equation}
is satisfied, then we have
\begin{itemize}
    \item [(i)] $$\lim_{n\to\infty}\int_{\mathbb{R}^N}\frac{|g(x,u_n)u_n-g(x,u_n-u)(u_n-u)-g(x,u)u|}{|x|^\beta}~\mathrm{d}x=0.$$
    \item [(ii)] $$\lim_{n\to\infty}\int_{\mathbb{R}^N}\frac{|G(x,u_n)-G(x,u_n-u)-G(x,u)|}{|x|^\beta}~\mathrm{d}x=0.$$
\end{itemize}
\end{lemma}
\begin{proof}
   By the hypotheses $(G1)$, $(G4)$ and Remark \ref{rem2.10}, we have for any $\vartheta_1,\vartheta_2\in\mathbb{R}$
   \begin{align*}|g(x,\vartheta_1+\vartheta_2)(\vartheta_1+\vartheta_2)-g(x,\vartheta_1)\vartheta_1|&=\bigg|\int_{0}^{1}\frac{\mathrm{d}}{\mathrm{d}s}\big(g(x,\vartheta_1+s\vartheta_2)(\vartheta_1+s\vartheta_2)\big)~\mathrm{d}s\bigg|\\
   &=\bigg|\int_{0}^{1}\big(\partial_u g(x,\vartheta_1+s\vartheta_2)\vartheta_2(\vartheta_1+s\vartheta_2)+g(x,\vartheta_1+s\vartheta_2)\vartheta_2\big)~\mathrm{d}s\bigg| \\
   &\leq 2 \int_{0}^{1}\Big[b_1|\vartheta_1+s\vartheta_2|^{N(k+1)-1}|\vartheta_2|+b_2|\vartheta_2|\Phi(\alpha_0|\vartheta_1+s\vartheta_2|^{N^\prime})\Big]~\mathrm{d}s\\
   &\leq 2^{N(k+1)-1}b_1 \int_{0}^{1}(|\vartheta_1|^{N(k+1)-1}+|s\vartheta_2|^{N(k+1)-1})|\vartheta_2|~\mathrm{d}s\\
&\qquad+2b_2\int_{0}^{1}|\vartheta_2|\Phi\big(\alpha_0(|\vartheta_1|+s|\vartheta_2|)^{N^\prime}\big)~\mathrm{d}s\\
&\leq 2^{N(k+1)-1}b_1\Big[|\vartheta_1|^{N(k+1)-1}|\vartheta_2|+|\vartheta_2|^{N(k+1)}\Big]+2b_2|\vartheta_2|\Phi\big(\alpha_0(|\vartheta_1|+|\vartheta_2|)^{N^\prime}\big).\end{align*}
Choose $\vartheta_1=u_n-u,\vartheta_2=u$ and $v_n=u_n-u$. Thus, from the above inequality, we have
\begin{align*}
 |g(x,v_n+u)(v_n+u)-g(x,v_n)v_n|\leq 2^{N(k+1)-1}b_1\Big[|v_n|^{N(k+1)-1}|u|+|u|^{N(k+1)}\Big]
     +2b_2|u|\Phi\big(\alpha_0(|v_n|+|u|)^{N^\prime}\big).   
\end{align*}
Further, set $h_n(x)=|g(x,v_n+u)(v_n+u)-g(x,v_n)v_n-g(x,u)u|$, then we have
\begin{equation}\label{eq4.10}
h_n(x)\leq 2^{N(k+1)-1}b_1|v_n|^{N(k+1)-1}|u|+(2^{N(k+1)-1}+1)b_1|u|^{N(k+1)}
+2b_2|u|\Phi\big(\alpha_0(|v_n|+|u|)^{N^\prime}\big)+b_2|u|\Phi(\alpha_0|u|^{N^\prime}).
\end{equation}
In view of remark \ref{remark2.4}, we have  $u_n\rightharpoonup u$ in $W^{1,N}(\mathbb{R}^N)$ as $n\to\infty$ and hence $\{u_n\}_n$ is bounded in $W^{1,N}(\mathbb{R}^N)$. Therefore, by Br\'ezis-Lieb lemma, we have
$$\|v_n\|^N_{W^{1,N}}=\|u_n\|^N_{W^{1,N}}-\|u\|^N_{W^{1,N}}+o_n(1)\leq \|u_n\|^N_{W^{1,N}}+o_n(1)~\text{as}~n\to\infty. $$
This together with \eqref{eq4.8} implies
\begin{equation}\label{eq4.11}
  \limsup_{n\to\infty}\|v_n\|^{N^\prime}_{W^{1,N}}<\bigg(1-\frac{\beta}{N}\bigg)\frac{\alpha_N}{2^{N^\prime}\alpha_0}.   
\end{equation}
Using weakly lower semicontinuity of the norm, we get
\begin{equation}\label{eq4.12}
 \|u\|^{N^\prime}_{W^{1,N}}\leq \liminf_{n\to\infty}\|u_n\|^{N^\prime}_{W^{1,N}}\leq \limsup_{n\to\infty}\|u_n\|^{N^\prime}_{W^{1,N}}<\bigg(1-\frac{\beta}{N}\bigg)\frac{\alpha_N}{2^{N^\prime}\alpha_0}.   
\end{equation}
From \eqref{eq4.11} and \eqref{eq4.12}, we can easily obtain
\begin{equation}\label{eq4.13}
  \limsup_{n\to\infty}\big\||v_n|+|u|\big\|^{N^\prime}_{W^{1,N}}<\bigg(1-\frac{\beta}{N}\bigg)\frac{\alpha_N}{\alpha_0}.   
\end{equation}
It follows from \eqref{eq4.13} that there exist $m_2>0$ and $n_0\in\mathbb{N}$ sufficiently large such that $\big\||v_n|+|u|\big\|^{N^\prime}_{W^{1,N}}<m_2<\big(1-\frac{\beta}{N}\big)\frac{\alpha_N}{\alpha_0}$, for all $n\ge n_0$. Choose $r\geq N$, with $r^\prime=\frac{r}{r-1}>1$ satisfying $\frac{1}{r}+\frac{1}{r^\prime}=1$ . Let $r^\prime$  close to 1 and $\alpha>\alpha_0$  close to $\alpha_0$ such that we still have $r^\prime\alpha\big\||v_n|+|u|\big\|^{N^\prime}_{W^{1,N}}<m_2<\big(1-\frac{\beta}{N}\big)\alpha_N$, for all $n\ge n_0$. Let us choose
 $$\widetilde{v}_n=\begin{cases}
        \frac{|v_n|+|u|}{\big\||v_n|+|u|\big\|_{W^{1,N}}}~\text{if either}~u_n\neq 0~\text{or}~u\neq0,\\
        0~~~~~~~~~~~~~~~~~~~\text{if}~u_n=u= 0.
    \end{cases}$$
Hence, by using H\"older's inequality, Lemma \ref{lem2.8} and Corollary \ref{cor2.9}, we have obtained from  \eqref{eq4.10} that for any $n\geq n_0$
\begin{align*}\int_{\mathbb{R}^N}\frac{h_n(x)}{|x|^\beta}~\mathrm{d}x &\leq 2^{N(k+1)-1}b_1\mathcal{S}^{-(N(k+1)-1)}_{(N(k+1)-1)N^\prime,\beta}\|v_n\|_{\mathbf{X}}^{N(k+1)-1} \|u\|_{L^N(\mathbb{R}^N,|x|^{-\beta})}+(2^{N(k+1)-1}+1)b_1\|u\|^{N(k+1)}_{L^{N(k+1)}(\mathbb{R}^N,|x|^{-\beta})}\\
&+b_2\Bigg[\int_{\mathbb{R}^N}\cfrac{\Phi\big(r^\prime\alpha_0|u|^{N^\prime}\big)}{|x|^\beta}~\mathrm{d}x\Bigg]^{\frac{1}{r^\prime}}\|u\|_{L^r(\mathbb{R}^N,|x|^{-\beta})}+2b_2\Bigg[\int_{\mathbb{R}^N}\cfrac{\Phi\big(r^\prime\alpha_0\big\||v_n|+|u|\big\|^{N^\prime}_{W^{1,N}}|\widetilde{v}_n|^{N^\prime}\big)}{|x|^\beta}~\mathrm{d}x\Bigg]^{\frac{1}{r^\prime}}\|u\|_{L^r(\mathbb{R}^N,|x|^{-\beta})}.\end{align*}
Consequently, we have
\begin{equation}
 \int_{\mathbb{R}^N}\frac{h_n(x)}{|x|^\beta}~\mathrm{d}x\leq C_3\Bigg[\bigg(\int_{\mathbb{R}^N}\frac{|u|^N}{|x|^\beta}~\mathrm{d}x\bigg)^{\frac{1}{N}}+\int_{\mathbb{R}^N}\frac{|u|^{N(k+1)}}{|x|^\beta}~\mathrm{d}x+\bigg(\int_{\mathbb{R}^N}\frac{|u|^r}{|x|^\beta}~\mathrm{d}x\bigg)^{\frac{1}{r}}\Bigg] <\infty,  
\end{equation}
where \begin{align*}C_3&=\text{max}\Bigg\{2^{N(k+1)-1}b_1\mathcal{S}^{-(N(k+1)-1)}_{(N(k+1)-1)N^\prime,\beta}\sup_{n\in\mathbb{N}}\|v_n\|_{\mathbf{X}}^{N(k+1)-1},(2^{N(k+1)-1}+1)b_1,\\
&\qquad b_2\Bigg[\int_{\mathbb{R}^N}\cfrac{\Phi\big(r^\prime\alpha_0|u|^{N^\prime}\big)}{|x|^\beta}~\mathrm{d}x\Bigg]^{\frac{1}{r^\prime}}+2b_2\Bigg[\sup_{n\geq n_0}\int_{\mathbb{R}^N}\cfrac{\Phi\big(r^\prime\alpha_0\big\||v_n|+|u|\big\|^{N^\prime}_{W^{1,N}}|\widetilde{v}_n|^{N^\prime}\big)}{|x|^\beta}~\mathrm{d}x\Bigg]^{\frac{1}{r^\prime}}\Bigg\},\end{align*}
which is finite ( thanks to remark \ref{rem2.10}, Lemma \ref{lem2.11} and the fact that $\{v_n\}_n$ is uniformly bounded in $\mathbf{X}$). It follows that $\{h_n(x)/|x|^\beta\}_{n\geq n_0}$ is in $L^1(\mathbb{R}^N)$. Moreover, it is not difficult to verify that $\{h_n(x)/|x|^\beta\}_{n\geq n_0}$ is uniformly integrable and tight over $\mathbb{R}^N$. Further, since $u_n\to u$ a.e. in $\mathbb{R}^N$ as $n\to\infty$, therefore $h_n(x)/|x|^\beta\to 0$ a.e. in $\mathbb{R}^N$ as $n\to\infty$. Finally, the conclusion of Lemma \ref{lem4.3}-$(i)$ follows immediately by Vitali's theorem. Arguing a similar procedure, we can also prove $(ii)$. This completes the proof.    
\end{proof}
Let $e\in\mathbf{X}$ stated as in Lemma \ref{lem3.2}, $C>0$ as in $(G5)$ and there exists $s>0$ satisfying $ps>N\xi$. Now, we define 
\begin{equation}\label{eq4.15}
    \lambda^\ast_2=\text{max}\Bigg\{1,\mathcal{D}^{\frac{s-N(k+1)}{N(k+1)}}_1,\mathcal{D}^{\frac{(\xi-p(k+1))(s-N(k+1))}{p(k+1)s-N(k+1)\xi}}_2\Bigg\},
\end{equation}
where $$ \mathcal{D}_1=\Bigg[\bigg(1-\frac{\beta}{N}\bigg)\frac{\alpha_N}{2^{N^\prime}\alpha_0}\Bigg]^{(1-N)(k+1)}\frac{1}{\Upsilon\big(\text{min}\{1,V_0\}\big)^{k+1}}   $$
and $$\mathcal{D}_2=\displaystyle\sum_{r\in\{p,N\}}\Bigg[\bigg(\frac{1}{r(k+1)}-\frac{1}{\xi}\bigg)\Bigg(\frac{\|e\|_{W^{1,r}_V}}{\|e\|_{L^\xi(\mathbb{R}^N,|x|^{-\beta})}}\Bigg)^{\frac{r(k+1)\xi}{\xi-r(k+1)}}\frac{a^{\frac{\xi}{\xi-r(k+1)}}}{{C}^{\frac{r(k+1)\xi}{\xi-r(k+1)}}}\Bigg]. $$
Choosing $\lambda_2>\lambda^\ast_2$, then it is easy to see that
  $\lambda_2^{\frac{N(k+1)}{s-N(k+1)}}<\lambda_2^{\frac{p(k+1)}{\xi-p(k+1)}}/\mathcal{D}_2$. Fix $c_0>0$ such that 
\begin{equation}\label{eq4.16}
  \lambda_2^{\frac{N(k+1)}{s-N(k+1)}}<c_0^{-1}<\lambda_2^{\frac{p(k+1)}{\xi-p(k+1)}}/\mathcal{D}_2.  
\end{equation}
Whenever $\lambda_2>\lambda^\ast_2$, we have $\mathcal{D}_1<\lambda_2^{\frac{N(k+1)}{s-N(k+1)}}$. This directly implies that
\begin{equation}\label{eq4.17}
    \lambda_2^{-{\frac{N(k+1)}{s-N(k+1)}}}<\Upsilon\big(\text{min}\{1,V_0\}\big)^{k+1} \Bigg[\bigg(1-\frac{\beta}{N}\bigg)\frac{\alpha_N}{2^{N^\prime}\alpha_0}\Bigg]^{(N-1)(k+1)}.
\end{equation}
Combining \eqref{eq4.16} and \eqref{eq4.17} together, we have
\begin{equation}\label{eq4.18}
    c_0 <\Upsilon\big(\text{min}\{1,V_0\}\big)^{k+1} \Bigg[\bigg(1-\frac{\beta}{N}\bigg)\frac{\alpha_N}{2^{N^\prime}\alpha_0}\Bigg]^{(N-1)(k+1)}.
\end{equation}
\begin{lemma}\label{lem4.4}
 Suppose $\{u_n\}_n$ be a $(PS)_c$ sequence for $J_{\lambda_1,\lambda_2}$ in $\mathbf{X}$ and $u$ be in $\mathbf{X}$. Let the assumptions of Lemma \ref{lem4.1} and  Lemma \ref{lem4.2} are satisfied. Then up to a subsequence $\nabla u_n\to \nabla u$ a.e. in $\mathbb{R}^N$ as $n\to\infty$.   
\end{lemma}
\begin{proof}
 Let $\{u_n\}_n$ be a $(PS)_c$ sequence for $J_{\lambda_1,\lambda_2}$ in $\mathbf{X}$ and $u$ be in $\mathbf{X}$. It follows that 
 \begin{equation}\label{eq4.19}
  J_{\lambda_1,\lambda_2}(u_n)\to c~\text{and}~ J^\prime_{\lambda_1,\lambda_2}(u_n)\to 0~~\text{in}~\mathbf{X^\ast}~\text{as} ~n\to\infty. 
 \end{equation}
 By Lemma \ref{lem4.1} and  Lemma \ref{lem4.2}, we have $\{u_n\}_n$ is bounded in $\mathbf{X}$ and \eqref{eq4.4} holds. Hence, up to a subsequence, $u_n\rightharpoonup u$ in $\mathbf{X}$ as $n\to\infty$. Arguing similarly as in Lemma \ref{lem3.4}, we can prove that $\{u_n\}_n$ and $u$ are nonnegative in $\mathbf{X}$. Moreover, by remark \ref{remark2.4} and Corollary \ref{cor2.3}, we obtain 
 \begin{equation}\label{eq4.20}
 \begin{cases}
   u_n\rightharpoonup u~\text{in}~W^{1,r}_V(\mathbb{R}^N)~\text{for}~r\in\{p,N\},\\
   u_n\to u ~\text{in}~L^\vartheta _{\text{loc}}(\mathbb{R}^N)~\text{for any}~\vartheta\in[1,\infty),\\
   \|u_n\|_{W^{1,p}_V}\to l_p,\|u_n\|_{W^{1,N}_V}\to l_N,~\text{where}~l_p,l_N\geq 0,\\
   u_n\to u~\text{a.e. in}~\mathbb{R}^N~\text{as}~n\to\infty.
 \end{cases}
 \end{equation} 
Let $l_p=l_N=0$, then $u_n\to 0$ in $\mathbf{X}$ as $n\to\infty$. Therefore, we infer from \eqref{eq4.19} that $0=J_{\lambda_1,\lambda_2}(0)=c>0$, which is a contradiction. If $l_p=0$ and $l_N\neq 0$ or $l_p\neq0$ and $l_N= 0$, then it is easy to see that $\nabla u_n\to 0$ a.e. in $\mathbb{R}^N$ as $n\to\infty$ and hence the proof is completed. Now we shall discuss the case when $l_p\neq0$ and $l_N\neq 0$. For this, define the functional $I_{\lambda_1,\lambda_2}:\mathbf{X}\to \mathbb{R} $  by
\begin{align*} I_{\lambda_1,\lambda_2}(u)&=\frac{M(l^p_p)}{p}\int_{\mathbb{R}^N}(|\nabla u|^p+V(x)|u|^p)~\mathrm{d}x+\frac{M(l^N_N)}{N}\int_{\mathbb{R}^N}(|\nabla u|^N+V(x)|u|^N)~\mathrm{d}x\\
&\qquad-\frac{\lambda_1}{q}\int_{\mathbb{R}^N}h(x)|u|^q~\mathrm{d}x-\lambda_2\int_{\mathbb{R}^N}\cfrac{G(x,u)}{|x|^\beta}~\mathrm{d}x.  \end{align*}
Clearly, $I_{\lambda_1,\lambda_2}\in \mathcal{C}^1(\mathbf{X},\mathbb{R})$. Choose $R>0$, $\psi\in \mathcal{C}^\infty(\mathbb{R}^N)$ with $0\leq \psi\leq 1$ in $\mathbb{R}^N$, $\psi\equiv1$ in $B_R$ and $\psi\equiv0$ in $B^c_{2R}$. Since $u_n\rightharpoonup u$ in $\mathbf{X}$ and $J^\prime_{\lambda_1,\lambda_2}(u_n)\to 0$ in $\mathbf{X}^\ast$ as $n\to\infty$, therefore we have 
\begin{equation}\label{eq4.21}
 \langle{J^\prime_{\lambda_1,\lambda_2}(u_n)-I^\prime_{\lambda_1,\lambda_2}(u),(u_n-u)\psi}\rangle=o_n(1)~\text{as}~n\to\infty.   
\end{equation}
By convexity and $(V1)$, we have
\begin{equation}\label{eq4.22}
    \begin{split}
     \big(|\nabla u_n|^{p-2}\nabla u_n-|\nabla u|^{p-2}\nabla u\big).\big(\nabla u_n-\nabla u\big)\geq 0~\text{a.e. in}~\mathbb{R}^N\\~\text{and}~
    V(x)\big(|u_n|^{p-2} u_n-| u|^{p-2} u\big)\big( u_n- u\big)\geq 0~\text{a.e. in}~\mathbb{R}^N
    \end{split}
\end{equation}
for any $n\in\mathbb{N}$. Thus, by Simon's inequality with $N\geq 2$, there exists $c_N>0$ such that
\begin{align*}c^{-1}_N M(l_N^N)\|u_n-u\|^N_{W^{1,N}_V(B_R)}&= c^{-1}_N M(l_N^N)\bigg[\int_{B_R}|\nabla u_n-\nabla u|^N~\mathrm{d}x+\int_{B_R}V(x)| u_n- u|^N~\mathrm{d}x\bigg]\\
&\leq\sum_{r\in\{p,N\}}\Bigg( M(l_r^r)\Bigg[\int_{B_R}\big(|\nabla u_n|^{r-2}\nabla u_n-|\nabla u|^{r-2}\nabla u\big).\big(\nabla u_n-\nabla u\big)~\mathrm{d}x\\
&\qquad +\int_{B_R}V(x)\big(u_n^{r-1}- u^{r-1} \big)\big( u_n- u\big)~\mathrm{d}x\Bigg] \Bigg)\\
&\leq \sum_{r\in\{p,N\}}\Bigg(M(l_r^r)\Bigg[\int_{\mathbb{R}^N}\big(|\nabla u_n|^{r-2}\nabla u_n-|\nabla u|^{r-2}\nabla u\big).\big(\nabla u_n-\nabla u\big)\psi~\mathrm{d}x\\
&\qquad +\int_{\mathbb{R}^N}V(x)\big(u_n^{r-1}- u^{r-1} \big)\big( u_n- u\big)\psi~\mathrm{d}x\Bigg]\Bigg)\\
& = -\sum_{r\in\{p,N\}} \Bigg(M(l_r^r)\Bigg[\int_{\mathbb{R}^N}\big(|\nabla u_n|^{r-2}\nabla u_n-|\nabla u|^{r-2}\nabla u\big).\nabla\psi( u_n-u)~\mathrm{d}x\Bigg]\Bigg)\\
&\qquad + \lambda_1\int_{\mathbb{R}^N}h(x)(u_n^{q-1}-u^{q-1})(u_n-u)\psi~\mathrm{d}x\end{align*} 
    \begin{equation}\label{eq4.23}
     \hspace{3.5cm}+\lambda_2\int_{\mathbb{R}^N}\cfrac{(g(x,u_n)-g(x,u))(u_n-u)\psi}{|x|^\beta}~\mathrm{d}x+o_n(1)~\text{as} ~n\to\infty,   
    \end{equation}
thanks to \eqref{eq4.21}. In view of \eqref{eq4.20}, by H\"older's inequality, we get for $r\in\{p,N\}$ 
\begin{align*}\bigg|\int_{\mathbb{R}^N}\big(|\nabla u_n|^{r-2}\nabla u_n-|\nabla u|^{r-2}\nabla u\big).\nabla\psi( u_n-u)~\mathrm{d}x\bigg|
&\leq \|\nabla\psi\|_\infty\big(\|\nabla u_n\|^{r-1}_r+\|\nabla u\|^{r-1}_r\big) \bigg(\int_{B_{2R}}|u_n-u|^r~\mathrm{d}x\bigg)^{\frac{1}{r}}\\
&\to 0~\text{as}~n\to\infty.\end{align*}
It follows that for $r\in\{p,N\}$, we have
\begin{equation}\label{eq4.24}
 \lim_{n\to\infty} \int_{\mathbb{R}^N}\big(|\nabla u_n|^{r-2}\nabla u_n-|\nabla u|^{r-2}\nabla u\big).\nabla\psi( u_n-u)~\mathrm{d}x=0.  
\end{equation}
Similarly, using H\"older's inequality, we have
$$\bigg|\int_{\mathbb{R}^N}h(x)(u_n^{q-1}-u^{q-1})(u_n-u)\psi~\mathrm{d}x\bigg| \leq \big(\|u_n\|^{q-1}_{q,h}+\|u\|^{q-1}_{q,h}\big)\|h\|^{\frac{1}{q}}_\eta \bigg(\int_{B_{2R}}|u_n-u|^N~\mathrm{d}x\bigg)^{\frac{1}{N}}\to 0~\text{as}~n\to\infty.$$
Hence, we get 
\begin{equation}\label{eq4.25}
 \lim_{n\to\infty}\int_{\mathbb{R}^N}h(x)(u_n^{q-1}-u^{q-1})(u_n-u)\psi~\mathrm{d}x =0.  
\end{equation}
Set $$ I_1=b_1\int_{B_{2R}}\cfrac{\big(|u_n|^{N(k+1)-1}+|u|^{N(k+1)-1}\big)|u_n-u|}{|x|^\beta}~\mathrm{d}x. $$
Choose $1<t<\frac{N}{\beta}$ and $1<s<\frac{N}{t\beta}$, then using  H\"older's inequality, one can easily obtain ( where we use $\frac{1}{\varsigma}+\frac{1}{\varsigma^\prime}=1$ for $\varsigma=t,s$)
\begin{align*}I_1 &\leq b_1\Bigg[\bigg(\int_{B_{2R}}\cfrac{|u_n|^{(N(k+1)-1)t}}{|x|^{\beta t}}~\mathrm{d}x\bigg)^{\frac{1}{t}}+\bigg(\int_{B_{2R}}\cfrac{|u|^{(N(k+1)-1)t}}{|x|^{\beta t}}~\mathrm{d}x\bigg)^{\frac{1}{t}}\Bigg]\|u_n-u\|_{L^{t^\prime}(B_{2R})} \\
&\leq b_1\big(\|u_n\|^{N(k+1)-1}_{L^{(N(k+1)-1)ts^\prime}{(B_{2R})}}+\|u\|^{N(k+1)-1}_{L^{(N(k+1)-1)ts^\prime}{(B_{2R})}}\big)\bigg(\int_{B_{2R}}\cfrac{\mathrm{d}x}{|x|^{\beta ts}}\bigg)^{\frac{1}{st}}\|u_n-u\|_{L^{t^\prime}(B_{2R})} \to 0~\text{as}~n\to\infty, \end{align*}
thanks to \eqref{eq4.20} and the fact that $\{u_n\}_n$ is bounded in $L^{(N(k+1)-1)ts^\prime}{(B_{2R})}$. Further, since \eqref{eq4.4} holds, therefore there exist $m_3>0$ and $n_0\in\mathbb{N}$ sufficiently large such that $\|u_n\|^{N^\prime}_{W^{1,N}}<m_3<\big(1-\frac{\beta}{N}\big)\frac{\alpha_N}{\alpha_0}$, for all $n\geq n_0$. Choose $r> 1$, with $r^\prime=\frac{r}{r-1}>1$ satisfying $\frac{1}{r}+\frac{1}{r^\prime}=1$ . Let $r^\prime$ close to 1 and $\alpha>\alpha_0$ close to $\alpha_0$ such that we still have $r^\prime\alpha\|u_n\|^{N^\prime}_{W^{1,N}}<m_3<\big(1-\frac{\beta}{N}\big)\alpha_N$, for all $n\geq n_0$. Let us assume that \eqref{eq3.333} is satisfied. Thus by H\"older's inequality, Corollary \ref{cor2.9} and \eqref{eq4.20}, we have  \begin{align*}I_2&=b_2 \int_{B_{2R}}\cfrac{\Phi\big(\alpha_0|u_n|^{N^\prime}\big)+\Phi\big(\alpha_0|u|^{N^\prime}\big)}{|x|^\beta}|u_n-u|~\mathrm{d}x\\&\leq b_2\Bigg(\Bigg\|\frac{\Phi\big(\alpha_0|u_n|^{N^\prime}\big)}{|x|^{\frac{\beta}{r^\prime}}}\Bigg\|_{r^\prime}+\Bigg\|\frac{\Phi\big(\alpha_0|u|^{N^\prime}\big)}{|x|^{\frac{\beta}{r^\prime}}}\Bigg\|_{r^\prime}\Bigg)\bigg(\int_{B_{2R}}\frac{|u_n-u|^r}{|x|^\beta}~\mathrm{d}x\bigg)^{\frac{1}{r}}\\
&\leq b_2 \Bigg(\Bigg[\int_{\mathbb{R}^N}\cfrac{\Phi\big(r^\prime\alpha_0\|u_n\|^{N^\prime}_{W^{1,N}}|\widetilde{u}_n|^{N^\prime}\big)}{|x|^\beta}~\mathrm{d}x\Bigg]^{\frac{1}{r^\prime}}+\Bigg[\int_{\mathbb{R}^N}\cfrac{\Phi\big(r^\prime\alpha_0|u|^{N^\prime}\big)}{|x|^\beta}~\mathrm{d}x\Bigg]^{\frac{1}{r^\prime}}\Bigg) \bigg(\int_{B_{2R}}\cfrac{\mathrm{d}x}{|x|^{\beta t}}\bigg)^{\frac{1}{rt}}\|u_n-u\|_{L^{rt^\prime}(B_{2R})} \\
& \leq C_4 \bigg(\int_{B_{2R}}\cfrac{\mathrm{d}x}{|x|^{\beta t}}\bigg)^{\frac{1}{rt}}\|u_n-u\|_{L^{rt^\prime}(B_{2R})}\to 0~\text{as}~n\to\infty,\end{align*}
where $$C_4= b_2 \Bigg(\Bigg[\sup_{n\geq n_0}\int_{\mathbb{R}^N}\cfrac{\Phi\big(r^\prime\alpha_0\|u_n\|^{N^\prime}_{W^{1,N}}|\widetilde{u}_n|^{N^\prime}\big)}{|x|^\beta}~\mathrm{d}x\Bigg]^{\frac{1}{r^\prime}}+\Bigg[\int_{\mathbb{R}^N}\cfrac{\Phi\big(r^\prime\alpha_0|u|^{N^\prime}\big)}{|x|^\beta}~\mathrm{d}x\Bigg]^{\frac{1}{r^\prime}}\Bigg)<\infty, $$
thanks to remark \ref{rem2.10} and Lemma \ref{lem2.11}. Moreover, by $(G1)$, we have
$$\bigg|\int_{\mathbb{R}^N}\cfrac{(g(x,u_n)-g(x,u))(u_n-u)\psi}{|x|^\beta}~\mathrm{d}x\bigg|\leq \int_{B_{2R}}\cfrac{(|g(x,u_n)|+|g(x,u)|)|u_n-u|}{|x|^\beta}~\mathrm{d}x \leq I_1+I_2\to 0~\text{as}~n\to\infty. $$
It follows that 
\begin{equation}\label{eq4.26}
  \lim_{n\to\infty} \int_{\mathbb{R}^N}\cfrac{(g(x,u_n)-g(x,u))(u_n-u)\psi}{|x|^\beta}~\mathrm{d}x=0.  
\end{equation}
Passing $n\to\infty$ in \eqref{eq4.23} and using \eqref{eq4.24}--\eqref{eq4.26}, we get $u_n\to u$ in $W^{1,N}(B_{R})$ as $n\to\infty$. Consequently, $\nabla u_n\to \nabla u$ in $L^{N}(B_{R})$ as $n\to\infty$, for all $R>0$. Hence, up to a subsequence $\nabla u_n\to \nabla u$ a.e. in $\mathbb{R}^N$ as $n\to\infty$. This completes the proof. 
\end{proof}
\begin{lemma}\label{lem4.5}
 Suppose $\lambda_1\in(0,\lambda_1^\ast]$, where $\lambda_1^\ast$ as in Lemma \ref{lem3.3} and $\lambda_2>\lambda^\ast_2$, with $\lambda^\ast_2$ defined as in \eqref{eq4.15}. Assume $c<c_0-\Gamma\lambda_1^{\sigma}$ as in Lemma \ref{lem4.2}, where $c_0>0$ satisfying \eqref{eq4.18} and $\Gamma$ as in Lemma \ref{lem4.1}. Then, every $(PS)_c$ sequence for $J_{\lambda_1,\lambda_2}$ in $\mathbf{X}$ converges strongly to a function $v_{\lambda_1,\lambda_2}\in\mathbf{X}$, which is a nontrivial solution of \eqref{main problem} at level $c$.
\end{lemma}
\begin{proof}
  Let $\{u_n\}_n$ be a $(PS)_c$ sequence for $J_{\lambda_1,\lambda_2}$ in $\mathbf{X}$. Thus \eqref{eq4.19} holds. By Lemma \ref{lem4.1}, $\{u_n\}_n$ is bounded in $\mathbf{X}$. Hence, without loss of generality, there exists $v_{\lambda_1,\lambda_2}\in\mathbf{X}$ such that up to a subsequence, $u_n\rightharpoonup v_{\lambda_1,\lambda_2}$ in $\mathbf{X}$ as $n\to\infty$. Repeating similarly as in Lemma \ref{lem3.4}, we can prove $\{u_n\}_n$ and $v_{\lambda_1,\lambda_2}$ are nonnegative in $\mathbf{X}$. Moreover, we deduce from remark \ref{remark2.4}, Corollary \ref{cor2.6}, Lemma \ref{lem2.8} and Lemma \ref{lem4.4} that
  \begin{equation}\label{eq4.27}
   \begin{cases}
       u_n\rightharpoonup v_{\lambda_1,\lambda_2} ~\text{in}~W_V^{1,r}(\mathbb{R}^N)~\text{and}~W^{1,r}(\mathbb{R}^N)~\text{for}~r\in\{p,N\},\\
       u_n\to v_{\lambda_1,\lambda_2} ~\text{in}~L^\tau(\mathbb{R}^N),~\text{for any}~\tau\in[p,p^\ast)\cup[N,\infty),\\
     u_n\to v_{\lambda_1,\lambda_2} ~\text{in}~L^q(\mathbb{R}^N,|x|^{-\beta}),~\text{for any}~q\in [N,\infty),\\
      u_n\to v_{\lambda_1,\lambda_2} ~\text{a.e. in}~\mathbb{R}^N,~
      \nabla u_n\to \nabla v_{\lambda_1,\lambda_2} ~\text{a.e. in}~\mathbb{R}^N,\\
      |\nabla u_n|^{r-2}\nabla u_n\rightharpoonup |\nabla v_{\lambda_1,\lambda_2}|^{r-2}\nabla v_{\lambda_1,\lambda_2}~\text{in}~L^{r^\prime}(\mathbb{R}^N)~\text{for}~r\in\{p,N\},\\
      V^{\frac{1}{r^\prime}} u_n^{r-1} \rightharpoonup V^{\frac{1}{r^\prime}} v_{\lambda_1,\lambda_2}^{r-1}~\text{in}~L^{r^\prime}(\mathbb{R}^N)~\text{for}~r\in\{p,N\}~\text{as}~n\to\infty.
      \end{cases}  
  \end{equation} 
However, from \eqref{eq4.27} and Br\'ezis-Lieb lemma, we have for $r\in\{p,N\}$ 
\begin{equation}\label{eq4.28}
    \begin{cases}
        \|u_n-v_{\lambda_1,\lambda_2}\|^r_{W_V^{1,r}}=\|u_n\|^{r}_{W_V^{1,r}}-\|v_{\lambda_1,\lambda_2}\|^r_{W_V^{1,r}}+o_n(1),\\
        \|u_n-v_{\lambda_1,\lambda_2}\|^r_{W^{1,r}}=\|u_n\|^{r}_{W^{1,r}}-\|v_{\lambda_1,\lambda_2}\|^r_{W^{1,r}}+o_n(1),
    \end{cases}
\end{equation}
as $n\to\infty$. Further, by \eqref{eq4.27}, we can assume that up to a subsequence, $\|u_n\|_{W^{1,p}_V}\to l_p$ and $\|u_n\|_{W^{1,N}_V}\to l_N$ as $n\to\infty$, where $l_p,l_N\geq 0$. Next, we obtain from \eqref{eq4.28} that 
 $$\|u_n-v_{\lambda_1,\lambda_2}\|^N_{W^{1,N}}=\|u_n\|^{N}_{W^{1,N}}-\|v_{\lambda_1,\lambda_2}\|^N_{W^{1,N}}+o_n(1)\leq \|u_n\|^{N}_{W^{1,N}}+o_n(1)~\text{as}~n\to\infty. $$   
Passing $n\to\infty$ in the above inequality and using \eqref{eq4.6}, we get
    \begin{align*}\limsup_{n\to\infty}\|u_n-v_{\lambda_1,\lambda_2}\|^{N^\prime}_{W^{1,N}}&\leq \limsup_{n\to\infty}\|u_n\|^{N^\prime}_{W^{1,N}}\leq \bigg(\frac{c+\Gamma\lambda_1^{\sigma}}{\Upsilon\big(\text{min}\{1,V_0\}\big)^{k+1}}\bigg)^{\frac{1}{(N-1)(k+1)}}<\bigg(1-\frac{\beta}{N}\bigg)\frac{\alpha_N}{2^{N^\prime}\alpha_0},\end{align*}
    thanks to \eqref{eq4.18}. It follows that
    \begin{equation}\label{eq4.29}
      \limsup_{n\to\infty}\|u_n-v_{\lambda_1,\lambda_2}\|^{N^\prime}_{W^{1,N}}<\bigg(1-\frac{\beta}{N}\bigg)\frac{\alpha_N}{2^{N^\prime}\alpha_0}<\bigg(1-\frac{\beta}{N}\bigg)\frac{\alpha_N}{\alpha_0}.  
    \end{equation}
 It follows from \eqref{eq4.29} that there exist $m_4>0$ and $n_0\in\mathbb{N}$ sufficiently large such that $\|u_n-v_{\lambda_1,\lambda_2}\|^{N^\prime}_{W^{1,N}}<m_4<\big(1-\frac{\beta}{N}\big)\frac{\alpha_N}{\alpha_0}$, for all $n\geq n_0$. Choose $r\geq N$, with $r^\prime=\frac{r}{r-1}>1$ satisfying $\frac{1}{r}+\frac{1}{r^\prime}=1$ . Let $r^\prime$ close to 1 and $\alpha>\alpha_0$ close to $\alpha_0$ such that we still have $r^\prime\alpha\|u_n-v_{\lambda_1,\lambda_2}\|^{N^\prime}_{W^{1,N}}<m_4<\big(1-\frac{\beta}{N}\big)\alpha_N$, for all $n\geq n_0$. Choose
     $$\widetilde{\nu}_n=\begin{cases}
        \frac{u_n-v_{\lambda_1,\lambda_2}}{\|u_n-v_{\lambda_1,\lambda_2}\|_{W^{1,N}}}~\text{if}~u_n-v_{\lambda_1,\lambda_2}\neq 0,\\
        0~~~~~~~~~~~~~~~~~~~~~~\text{if}~u_n-v_{\lambda_1,\lambda_2}=0.
    \end{cases}$$
 Now, by $(G1)$, H\"older's inequality, Corollary \ref{cor2.9} and \eqref{eq4.27}, we obtain 
\begin{align*}\bigg|\int_{\mathbb{R}^N}\cfrac{g(x,u_n-v_{\lambda_1,\lambda_2})(u_n-v_{\lambda_1,\lambda_2})}{|x|^\beta}~\mathrm{d}x\bigg|&\leq b_1\|u_n-v_{\lambda_1,\lambda_2}\|^{N(k+1)}_{L^{N(k+1)}(\mathbb{R}^N,|x|^{-\beta})}\\
&\qquad+b_2\Bigg\|\frac{\Phi\big(\alpha_0|u_n-v_{\lambda_1,\lambda_2}|^{N^\prime}\big)}{|x|^{\frac{\beta}{r^\prime}}}\Bigg\|_{r^\prime}\Bigg\| \frac{u_n-v_{\lambda_1,\lambda_2}}{|x|^{\frac{\beta}{r}}}\Bigg\|_{r} \\
&\leq b_1\|u_n-v_{\lambda_1,\lambda_2}\|^{N(k+1)}_{L^{N(k+1)}(\mathbb{R}^N,|x|^{-\beta})}\\
&\quad+b_2 \Bigg[\int_{\mathbb{R}^N}\cfrac{\Phi\big(r^\prime\alpha_0|u_n-v_{\lambda_1,\lambda_2}|^{N^\prime}\big)}{|x|^\beta}~\mathrm{d}x\Bigg]^{\frac{1}{r^\prime}}\|u_n-v_{\lambda_1,\lambda_2}\|_{L^r(\mathbb{R}^N,|x|^{-\beta})}\\
&= b_1\|u_n-v_{\lambda_1,\lambda_2}\|^{N(k+1)}_{L^{N(k+1)}(\mathbb{R}^N,|x|^{-\beta})}\\
&\quad+b_2\Bigg[\int_{\mathbb{R}^N}\cfrac{\Phi\big(r^\prime\alpha_0\|u_n-v_{\lambda_1,\lambda_2}\|_{{W^{1,N}}}^{N^\prime}|\widetilde{\nu}_n|^{N^\prime}\big)}{|x|^\beta}~\mathrm{d}x\Bigg]^{\frac{1}{r^\prime}}\|u_n-v_{\lambda_1,\lambda_2}\|_{L^r(\mathbb{R}^N,|x|^{-\beta})} \\
&\leq b_1\|u_n-v_{\lambda_1,\lambda_2}\|^{N(k+1)}_{L^{N(k+1)}(\mathbb{R}^N,|x|^{-\beta})}+C_5\|u_n-v_{\lambda_1,\lambda_2}\|_{L^r(\mathbb{R}^N,|x|^{-\beta})}\\
&\quad\to 0~\text{as}~n\to\infty,  
\end{align*}
where $$C_5=b_2\Bigg[\sup_{n\geq n_0}\int_{\mathbb{R}^N}\cfrac{\Phi\big(r^\prime\alpha_0\|u_n-v_{\lambda_1,\lambda_2}\|_{{W^{1,N}}}^{N^\prime}|\widetilde{\nu}_n|^{N^\prime}\big)}{|x|^\beta}~\mathrm{d}x\Bigg]^{\frac{1}{r^\prime}}<\infty,$$
due to Lemma \ref{lem2.11} and remark \ref{rem2.10}.
It follows that 
\begin{equation}\label{eq4.30}
  \lim_{n\to\infty}  \int_{\mathbb{R}^N}\cfrac{g(x,u_n-v_{\lambda_1,\lambda_2})(u_n-v_{\lambda_1,\lambda_2})}{|x|^\beta}~\mathrm{d}x=0.
\end{equation}
In view of $(G1)$, $(G2)$ and Fatou's lemma, we have
\begin{equation}\label{eq4.31}
    \int_{\mathbb{R}^N} g(x,v_{\lambda_1,\lambda_2})v_{\lambda_1,\lambda_2}~\mathrm{d}x\leq\liminf_{n\to\infty} \int_{\mathbb{R}^N} g(x,u_n)v_{\lambda_1,\lambda_2}~\mathrm{d}x.
\end{equation}
Similar to the prove of \eqref{eq3.19}, we can easily obtain
 \begin{equation}\label{eq4.32}
   \lim_{n\to\infty} \int_{\mathbb{R}^N}h(x)u_n^{q-1}(u_n-v_{\lambda_1,\lambda_2})~\mathrm{d}x=0. 
 \end{equation}
 Now, by using \eqref{eq4.19}, \eqref{eq4.27}, \eqref{eq4.31} and \eqref{eq4.32}, we have
 \begin{align*}0&=\lim_{n\to\infty}\langle{J^\prime_{\lambda_1,\lambda_2}(u_n),u_n-v_{\lambda_1,\lambda_2}}\rangle \\
 &=\sum_{r\in\{p,N\}}\Bigg( M(l^r_r) \lim_{n\to\infty}\Bigg[\|u_n\|^r_{W^{1,r}_V}-\int_{\mathbb{R}^N}\big(|\nabla u_n|^{r-2}\nabla u_n. \nabla v_{\lambda_1,\lambda_2}+V(x)u_n^{r-1}v_{\lambda_1,\lambda_2}\big)~\mathrm{d}x\Bigg]\Bigg)\\
 &\qquad-\lambda_1 \lim_{n\to\infty} \int_{\mathbb{R}^N}h(x)u_n^{q-1}(u_n-v_{\lambda_1,\lambda_2})~\mathrm{d}x -\lambda_2 \lim_{n\to\infty} \int_{\mathbb{R}^N}\cfrac{g(x,u_n)(u_n-v_{\lambda_1,\lambda_2})}{|x|^\beta}~\mathrm{d}x\\
 &\geq \sum_{r\in\{p,N\}}\bigg( M(l^r_r)\big(l^r_r-\|v_{\lambda_1,\lambda_2}\|^r_{W^{1,r}_V}\big)\bigg)-\lambda_2 \lim_{n\to\infty} \int_{\mathbb{R}^N}\cfrac{g(x,u_n)u_n-g(x,v_{\lambda_1,\lambda_2})v_{\lambda_1,\lambda_2}}{|x|^\beta}~\mathrm{d}x. \end{align*}
 It follows from Lemma \ref{lem4.3}, \eqref{eq4.28} and \eqref{eq4.30} that
 \begin{align*} \sum_{r\in\{p,N\}}\Big( M(l^r_r)\lim_{n\to\infty} \|u_n-v_{\lambda_1,\lambda_2}\|^r_{W^{1,r}_V}\Big)\leq \lambda_2  \lim_{n\to\infty}  \int_{\mathbb{R}^N}\cfrac{g(x,u_n-v_{\lambda_1,\lambda_2})(u_n-v_{\lambda_1,\lambda_2})}{|x|^\beta}~\mathrm{d}x=0.\end{align*}
 Consequently, we obtain
 \begin{equation}\label{eq4.33}
  M(l^p_p)\lim_{n\to\infty} \|u_n-v_{\lambda_1,\lambda_2}\|^p_{W^{1,p}_V}+M(l^N_N)\lim_{n\to\infty} \|u_n-v_{\lambda_1,\lambda_2}\|^N_{W^{1,N}_V}=0.   
 \end{equation}
\textit{Case-1:} Let $l_p=l_N=0$. This situation cannot occur. Indeed, if $l_p=l_N=0$, then $u_n\to 0$ in $\mathbf{X}$ as $n\to\infty$. It follows from \eqref{eq4.19} that $0=J^\prime_{\lambda_1,\lambda_2}(0)=c>0$, which is absurd situation.\\
\textit{Case-2}: Let either $l_p=0$ and $l_N\neq 0$ or $l_p\neq0$ and $l_N= 0$. We only discuss the situation when $l_p=0$ and $l_N\neq 0$. Since $l_p=0$ and $l_N\neq 0$, therefore it follows from \eqref{eq4.33} and $(M)$ that $u_n\to v_{\lambda_1,\lambda_2} $ in $W^{1,N}_V(\mathbb{R}^N)$ as $n\to\infty$. Consequently, we have $\|u_n\|_{W^{1,N}_V}\to \|v_{\lambda_1,\lambda_2}\|_{W^{1,N}_V}$ in $\mathbb{R}$ as $n\to\infty$. Thus, we have
\begin{equation}\label{eq4.34}
 \|v_{\lambda_1,\lambda_2}\|_{W^{1,N}_V}=l_N>0.   
\end{equation}

On the other hand, since $l_p=0$, therefore $u_n\to 0$ in $W^{1,p}_V(\mathbb{R}^N)$ as $n\to\infty$. Hence, we have that $\nabla u_n\to 0$ in $L^p(\mathbb{R}^N)$ as $n\to\infty$ and $ u_n\to 0$ in $L^p_V(\mathbb{R}^N)$ as $n\to\infty$. Now, going further subsequence, $\nabla u_n\to 0$ a.e. in $\mathbb{R}^N$ as $n\to\infty$ and $u_n\to 0$ a.e. in $\mathbb{R}^N$ as $n\to\infty$. It follows that $\nabla v_{\lambda_1,\lambda_2}=0$ a.e. in $\mathbb{R}^N$ and $ v_{\lambda_1,\lambda_2}=0$ a.e. in $\mathbb{R}^N$, which contradicts \eqref{eq4.34}. Thus, this situation is also impossible.\\
\textit{Case-3}: Let $l_p\neq 0$ and $l_N\neq 0$. Then, it follows immediately from \eqref{eq4.33} and $(M)$ that $u_n\to v_{\lambda_1,\lambda_2}$ in $W^{1,r}_V(\mathbb{R}^N)$ as $n\to\infty$, for $r\in\{p,N\}$. Hence, we deduce that $u_n\to v_{\lambda_1,\lambda_2}$ in $\mathbf{X}$ as $n\to\infty$. Moreover, we get $J_{\lambda_1,\lambda_2}(v_{\lambda_1,\lambda_2})=c>0$ and $J^\prime_{\lambda_1,\lambda_2}(v_{\lambda_1,\lambda_2})=0$. This shows that $v_{\lambda_1,\lambda_2}\neq 0$ and $v_{\lambda_1,\lambda_2}$ is a nontrivial nonnegative solution of \eqref{main problem}. Hence, the proof is completed.
\end{proof}

Let $\lambda_1\in(0,\lambda_1^\ast]$, where $\lambda_1^\ast$ as in Lemma \ref{lem3.3} and $\lambda_2>0$ be fixed. Note that the mountain pass geometrical structures of the energy functional $J_{\lambda_1,\lambda_2}$ are satisfied (see Lemma \ref{lem3.2} \& Lemma \ref{lem3.3}). Now, define a mountain pass level $c_{\lambda_1,\lambda_2}$ for $J_{\lambda_1,\lambda_2}$ by
\begin{equation}\label{eq4.35}   c_{\lambda_1,\lambda_2}=\inf_{\gamma\in\Lambda}\max_{\tau\in[0,1] } J_{\lambda_1,\lambda_2}(\gamma(\tau)),
\end{equation}
where $\Lambda=\{\gamma\in\mathcal{C}([0,1],\mathbf{X}):~\gamma(0)=0,~J_{\lambda_1,\lambda_2}(\gamma(1))<0\}$. Using Lemma \ref{lem3.3}, it is obvious that $c_{\lambda_1,\lambda_2}>0$. 

\par To apply the mountain pass theorem, which provides us the second independent solution for \eqref{main problem}, we shall show that $c_{\lambda_1,\lambda_2}$ falls into the range of validity of the $(PS)_c$ condition given in Lemma \ref{lem4.5}.
\begin{lemma}\label{lem4.6}
Let $\lambda_2>\lambda_2^\ast$, where $\lambda_2^\ast$ is defined as in \eqref{eq4.15}. Then, there exists $\tilde{\lambda}_1\in (0,\lambda_1^\ast]$, depending on $\lambda_2$, such that $c_{\lambda_1,\lambda_2}<c_0-\Gamma\lambda_1^{\sigma}$ for any $\lambda_1\in(0,\tilde{\lambda}_1]$, with $c_0$ satisfying \eqref{eq4.16} and $\sigma,\Gamma$ as in Lemma \ref{lem4.1}.
\end{lemma}
\begin{proof}
Let $\lambda_2>\lambda_2^\ast$ be fixed and $e\in\mathbf{X}$ as in Lemma \ref{lem3.2}. Further, assume $\Bar{\lambda}_1\in(0,\lambda_1^\ast]$ sufficiently small enough such that $c_0\geq 2\Gamma\lambda_1^{\sigma}$ holds for all $\lambda_1\in(0,\Bar{\lambda}_1]$. It is easy to see that $J_{\lambda_1,\lambda_2}(s e)\to 0$ as $s\to 0^+$. Thus, we can choose $s_0$ very small such that 
\begin{equation}\label{eq4.36}
    \sup_{s\in[0,s_0]} J_{\lambda_1,\lambda_2}(s e)<\frac{c_0}{2}\leq c_0-\Gamma\lambda_1^{\sigma},~\text{for any}~\lambda_1\in(0,\Bar{\lambda}_1].
\end{equation}
Notice that $\lambda_2>1$ and $N>p$, therefore by using $(G5)$, we obtain
\begin{align*} \sup_{s\geq s_0} J_{\lambda_1,\lambda_2}&(s e)\leq \sup_{s\geq s_0}\bigg(\sum_{r\in\{p,N\}}\Bigg[\frac{as^{r(k+1)}}{r(k+1)}\|e\|^{r(k+1)}_{W^{1,r}_V}\Bigg]-\frac{\lambda_1 s^q}{q}\|e\|^q_{q,h}-\frac{2\lambda_2 C s^\xi}{\xi}\|e\|^\xi_{L^\xi(\mathbb{R}^N,|x|^{-\beta})}\bigg)\\
&\leq \sum_{r\in\{p,N\}}\Bigg[\max_{s\geq 0}\bigg(\frac{a s^{r(k+1)}}{r(k+1)}\|e\|^{r(k+1)}_{W^{1,r}_V}-\frac{\lambda_2 C s^\xi}{\xi}\|e\|^\xi_{L^\xi(\mathbb{R}^N,|x|^{-\beta})}\bigg)\Bigg] -\frac{\lambda_1 s_0^q}{q}\|e\|^q_{q,h}\\
&= \sum_{r\in\{p,N\}}\Bigg[\bigg(\frac{1}{r(k+1)}-\frac{1}{\xi}\bigg)\Bigg(\frac{\|e\|_{W^{1,r}_V}}{\|e\|_{L^\xi(\mathbb{R}^N,|x|^{-\beta})}}\Bigg)^{\frac{r(k+1)\xi}{\xi-r(k+1)}}\frac{a^{\frac{\xi}{\xi-r(k+1)}}}{(C\lambda_2)^{\frac{r(k+1)\xi}{\xi-r(k+1)}}}\Bigg] -\frac{\lambda_1 s_0^q}{q}\|e\|^q_{q,h} \\
&< {\lambda_2}^{-\frac{p(k+1)\xi}{\xi-p(k+1)}}\Bigg(\sum_{r\in\{p,N\}}\Bigg[\bigg(\frac{1}{r(k+1)}-\frac{1}{\xi}\bigg)\Bigg(\frac{\|e\|_{W^{1,r}_V}}{\|e\|_{L^\xi(\mathbb{R}^N,|x|^{-\beta})}}\Bigg)^{\frac{r(k+1)\xi}{\xi-r(k+1)}}\frac{a^{\frac{\xi}{\xi-r(k+1)}}}{{C}^{\frac{r(k+1)\xi}{\xi-r(k+1)}}}\Bigg]\Bigg)-\frac{\lambda_1 s_0^q}{q}\|e\|^q_{q,h}\\
&<c_0-\frac{\lambda_1 s_0^q}{q}\|e\|^q_{q,h},
\end{align*}
thanks to \eqref{eq4.16}. Choose $\Bar{\Bar{\lambda}}_1\in(0,\lambda_1^\ast]$ very small such that
  $$-\frac{\lambda_1 s_0^q}{q}\|e\|^q_{q,h}<-\Gamma\lambda_1^{\sigma},~\text{for any}~\lambda_1\in(0,\Bar{\Bar{\lambda}}_1].$$  
It follows that
\begin{equation}\label{eq4.38}
 \sup_{s\geq s_0} J_{\lambda_1,\lambda_2}(s e)< c_0-\Gamma\lambda_1^{\sigma},~\text{for any}~\lambda_1\in(0,\Bar{\Bar{\lambda}}_1].  
\end{equation}
Let us choose $\tilde{\lambda}_1=\min\{\Bar{\lambda}_1,\Bar{\Bar{\lambda}}_1\}$ and the path $\gamma(s)=s e$ belongs to $\Lambda$, where $s\in[0,1]$. Now, from \eqref{eq4.36} and \eqref{eq4.38}, we obtain
  $$ c_{\lambda_1,\lambda_2}\leq\max_{s\in[0,1] } J_{\lambda_1,\lambda_2}(\gamma(s))\leq \max_{s\geq 0 } J_{\lambda_1,\lambda_2}(\gamma(s))< c_0-\Gamma\lambda_1^{\sigma},~\text{for any}~\lambda_1\in(0,\tilde{\lambda}_1].$$    
This completes the proof.
\end{proof}
\begin{proof}[\textbf{Proof of Theorem \ref{thm1.2}}.]
The energy functional $J_{\lambda_1,\lambda_2}$ enjoys the mountain pass geometric structures for all $\lambda_1\in(0,\lambda_1^\ast]$ and for all $\lambda_2$, thanks to Lemma \ref{lem3.2} and Lemma \ref{lem3.3}. Further, by Lemma \ref{lem4.5} and Lemma \ref{lem4.6}, we have for $\lambda_2>\lambda_2^\ast$, there exists $\tilde{\lambda}_1\leq {\lambda}^\ast_1$ such that $J_{\lambda_1,\lambda_2}$ admits a critical point $v_{\lambda_1,\lambda_2}\in\mathbf{X}$ at the mountain pass level $c_{\lambda_1,\lambda_2}>0$, estimated as in Lemma \ref{lem4.6} for all $\lambda_1\in(0,\tilde{\lambda}_1]$. Also, by Lemma \ref{lem4.5}, we can see that $v_{\lambda_1,\lambda_2}\in\mathbf{X}$ is a nontrivial nonnegative solution of \eqref{main problem}, which is independent of the solution $u_{\lambda_1,\lambda_2}$ of \eqref{main problem} obtained in Theorem \ref{thm1.1}, since $J_{\lambda_1,\lambda_2}(u_{\lambda_1,\lambda_2})=m_{\lambda_1,\lambda_2}<0<c_{\lambda_1,\lambda_2}=J_{\lambda_1,\lambda_2}(v_{\lambda_1,\lambda_2})$. This finishes the proof of Theorem \ref{thm1.2}.
\end{proof}
\section*{Acknowledgements} DKM would like to express his heartfelt gratitude to the DST INSPIRE Fellowship DST/INSPIRE/03/2019/000265 supported by Govt. of India.TM gratefully thanks the assistance of the Start up Research Grant from DST-SERB, sanction no. SRG/2022/000524. AS was supported by the DST-INSPIRE Grant DST/INSPIRE/04/2018/002208 sponsored by Govt. of India.
 
\bibliography{ref}

\begin{thebibliography}{10}

\bibitem{MR2669653}
Adimurthi and Y.~Yang.
\newblock An interpolation of {H}ardy inequality and {T}rundinger-{M}oser
  inequality in {$\Bbb R^N$} and its applications.
\newblock {\em Int. Math. Res. Not. IMRN}, 2010(13):2394--2426, 2010.

\bibitem{do1997n}
J.~M. B.~do \'{O}.
\newblock {$N$}-{L}aplacian equations in {$\mathbb{R}^N$} with critical growth.
\newblock {\em Abstr. Appl. Anal.}, 2(3-4):301--315, 1997.

\bibitem{MR3732173}
A.~Bahrouni and V.~D. R\u{a}dulescu.
\newblock On a new fractional {S}obolev space and applications to nonlocal
  variational problems with variable exponent.
\newblock {\em Discrete Contin. Dyn. Syst. Ser. S}, 11(3):379--389, 2018.

\bibitem{MR3454625}
R.~Bartolo, A.~M. Candela, and A.~Salvatore.
\newblock Multiplicity results for a class of asymptotically {$p$}-linear
  equations on {$\Bbb{R}^N$}.
\newblock {\em Commun. Contemp. Math.}, 18(1):1550031, 24, 2016.

\bibitem{MR1230930}
W.~Beckner.
\newblock Sharp {S}obolev inequalities on the sphere and the
  {M}oser-{T}rudinger inequality.
\newblock {\em Ann. of Math. (2)}, 138(1):213--242, 1993.

\bibitem{MR4596285}
Z.~Binlin, X.~Han, and N.~V. Thin.
\newblock Schr\"{o}dinger-{K}irchhof-type problems involving the fractional
  {$p$}-{L}aplacian with exponential growth.
\newblock {\em Appl. Anal.}, 102(7):1942--1974, 2023.

\bibitem{MR4289110}
J.~L. Carvalho, G.~M. Figueiredo, M.~F. Furtado, and E.~Medeiros.
\newblock On a zero-mass {$(N,q)$}-{L}aplacian equation in {$\Bbb{R}^N$} with
  exponential critical growth.
\newblock {\em Nonlinear Anal.}, 213:Paper No. 112488, 14, 2021.

\bibitem{MR1989228}
S.-Y.~A. Chang and P.~C. Yang.
\newblock The inequality of {M}oser and {T}rudinger and applications to
  conformal geometry.
\newblock {\em Comm. Pure Appl. Math.}, 56(8):1135--1150, 2003.

\bibitem{MR3735478}
C.~Chen.
\newblock Infinitely many solutions for {$N$}-{K}irchhoff equation with
  critical exponential growth in {$\Bbb R^N$}.
\newblock {\em Mediterr. J. Math.}, 15(1):Paper No. 4, 20, 2018.

\bibitem{MR3483510}
C.~Chen and Q.~Chen.
\newblock Infinitely many solutions for {$p$}-{K}irchhoff equation with
  concave-convex nonlinearities in {$\Bbb R^N$}.
\newblock {\em Math. Methods Appl. Sci.}, 39(6):1493--1504, 2016.

\bibitem{MR4149293}
S.~Chen, A.~Fiscella, P.~Pucci, and X.~Tang.
\newblock Coupled elliptic systems in {$\Bbb {R}^N$} with {$(p,N)$} {L}aplacian
  and critical exponential nonlinearities.
\newblock {\em Nonlinear Anal.}, 201:112066, 14, 2020.

\bibitem{MR2788328}
M.~de~Souza.
\newblock On a singular elliptic problem involving critical growth in {$\Bbb
  R^N$}.
\newblock {\em NoDEA Nonlinear Differential Equations Appl.}, 18(2):199--215,
  2011.

\bibitem{MR3042696}
M.~de~Souza and J.~a.~M. do~\'{O}.
\newblock On singular {T}rudinger-{M}oser type inequalities for unbounded
  domains and their best exponents.
\newblock {\em Potential Anal.}, 38(4):1091--1101, 2013.

\bibitem{MR2838280}
J.~a.~M. do~\'{O} and M.~de~Souza.
\newblock On a class of singular {T}rudinger-{M}oser type inequalities and its
  applications.
\newblock {\em Math. Nachr.}, 284(14-15):1754--1776, 2011.

\bibitem{MR3145759}
J.~A.~M. do~\'{O}, M.~de~Souza, E.~de~Medeiros, and U.~Severo.
\newblock An improvement for the {T}rudinger-{M}oser inequality and
  applications.
\newblock {\em J. Differential Equations}, 256(4):1317--1349, 2014.

\bibitem{MR3896668}
G.~M. Figueiredo and F.~B.~M. Nunes.
\newblock Existence of positive solutions for a class of quasilinear elliptic
  problems with exponential growth via the {N}ehari manifold method.
\newblock {\em Rev. Mat. Complut.}, 32(1):1--18, 2019.

\bibitem{MR4258779}
A.~Fiscella and P.~Pucci.
\newblock {$(p,N)$} equations with critical exponential nonlinearities in
  {$\Bbb R^N$}.
\newblock {\em J. Math. Anal. Appl.}, 501(1):Paper No. 123379, 25, 2021.

\bibitem{MR4334134}
Y.~Gao and L.~Liu.
\newblock Multiple solutions for {$N$}-{K}irchhoff type problem in {$\Bbb
  R^N$}.
\newblock {\em Appl. Math. Lett.}, 125:Paper No. 107743, 7, 2022.

\bibitem{MR2198572}
J.~Giacomoni and K.~Sreenadh.
\newblock A multiplicity result to a nonhomogeneous elliptic equation in whole
  space {$\Bbb R^2$}.
\newblock {\em Adv. Math. Sci. Appl.}, 15(2):467--488, 2005.

\bibitem{MR4622428}
S.~Gupta and G.~Dwivedi.
\newblock Existence and multiplicity of solutions to {$N$}-{K}irchhoff
  equations with critical exponential growth and a perturbation term.
\newblock {\em Complex Var. Elliptic Equ.}, 68(8):1332--1360, 2023.

\bibitem{kirchhoff1883vorlesungen}
G.~Kirchhoff.
\newblock Vorlesungen uber.
\newblock {\em Mechanik, Leipzig, Teubner}, 1883.

\bibitem{MR2863858}
N.~Lam and G.~Lu.
\newblock Existence and multiplicity of solutions to equations of
  {$N$}-{L}aplacian type with critical exponential growth in {$\Bbb R^N$}.
\newblock {\em J. Funct. Anal.}, 262(3):1132--1165, 2012.

\bibitem{MR2980499}
N.~Lam and G.~Lu.
\newblock Sharp {M}oser-{T}rudinger inequality on the {H}eisenberg group at the
  critical case and applications.
\newblock {\em Adv. Math.}, 231(6):3259--3287, 2012.

\bibitem{MR3130507}
N.~Lam, G.~Lu, and H.~Tang.
\newblock Sharp subcritical {M}oser-{T}rudinger inequalities on {H}eisenberg
  groups and subelliptic {PDE}s.
\newblock {\em Nonlinear Anal.}, 95:77--92, 2014.

\bibitem{MR3797740}
J.~Li, G.~Lu, and M.~Zhu.
\newblock Concentration-compactness principle for {T}rudinger-{M}oser
  inequalities on {H}eisenberg groups and existence of ground state solutions.
\newblock {\em Calc. Var. Partial Differential Equations}, 57(3):Paper No. 84,
  26, 2018.

\bibitem{MR3316612}
Q.~Li and Z.~Yang.
\newblock Multiple solutions for {$N$}-{K}irchhoff type problems with critical
  exponential growth in {$\Bbb{R}^N$}.
\newblock {\em Nonlinear Anal.}, 117:159--168, 2015.

\bibitem{MR2400264}
Y.~Li and B.~Ruf.
\newblock A sharp {T}rudinger-{M}oser type inequality for unbounded domains in
  {$\Bbb R^n$}.
\newblock {\em Indiana Univ. Math. J.}, 57(1):451--480, 2008.

\bibitem{MR2794422}
C.~Liu and Y.~Zheng.
\newblock Existence of nontrivial solutions for {$p$}-{L}aplacian equations in
  {${\bf R}^N$}.
\newblock {\em J. Math. Anal. Appl.}, 380(2):669--679, 2011.

\bibitem{MR0273370}
T.-S. Liu and A.~Van~Rooij.
\newblock Sums and intersections of normed linear spaces.
\newblock {\em Math. Nachr.}, 42:29--42, 1969.

\bibitem{MR4487653}
Y.~Liu and C.~Liu.
\newblock The ground state solutions for {K}irchhoff-{S}chr\"{o}dinger type
  equations with singular exponential nonlinearities in {$\Bbb{R}^N$}.
\newblock {\em Chinese Ann. Math. Ser. B}, 43(4):549--566, 2022.

\bibitem{MR2488689}
J.~A. Marcos~do \'{O}, E.~Medeiros, and U.~Severo.
\newblock On a quasilinear nonhomogeneous elliptic equation with critical
  growth in {$\Bbb R^N$}.
\newblock {\em J. Differential Equations}, 246(4):1363--1386, 2009.

\bibitem{MR3412392}
P.~Pucci, M.~Xiang, and B.~Zhang.
\newblock Multiple solutions for nonhomogeneous {S}chr\"{o}dinger-{K}irchhoff
  type equations involving the fractional {$p$}-{L}aplacian in {$\Bbb{R}^N$}.
\newblock {\em Calc. Var. Partial Differential Equations}, 54(3):2785--2806,
  2015.

\bibitem{MR3987388}
M.~Xiang, B.~Zhang, and D.~Repov\v{s}.
\newblock Existence and multiplicity of solutions for fractional
  {S}chr\"{o}dinger-{K}irchhoff equations with {T}rudinger-{M}oser
  nonlinearity.
\newblock {\em Nonlinear Anal.}, 186:74--98, 2019.

\bibitem{MR2873855}
Y.~Yang.
\newblock Existence of positive solutions to quasi-linear elliptic equations
  with exponential growth in the whole {E}uclidean space.
\newblock {\em J. Funct. Anal.}, 262(4):1679--1704, 2012.

\bibitem{MR3483063}
Y.~Yang and K.~Perera.
\newblock {$(N,q)$}-{L}aplacian problems with critical {T}rudinger-{M}oser
  nonlinearities.
\newblock {\em Bull. Lond. Math. Soc.}, 48(2):260--270, 2016.

\bibitem{MR3836136}
C.~Zhang and L.~Chen.
\newblock Concentration-compactness principle of singular {T}rudinger-{M}oser
  inequalities in {$\Bbb R^n$} and {$n$}-{L}aplace equations.
\newblock {\em Adv. Nonlinear Stud.}, 18(3):567--585, 2018.

\bibitem{MR4411714}
J.~Zhang.
\newblock Existence results for a {K}irchhoff-type equations involving the
  fractional {$p_1(x)$} \& {$p_2(x)$}-{L}aplace operator.
\newblock {\em Collect. Math.}, 73(2):271--293, 2022.

\bibitem{MR4365176}
Y.~Zhang and Y.~Yang.
\newblock Positive solutions for semipositone {$(p,N)$}-{L}aplacian problems
  with critical {T}rudinger-{M}oser nonlinearities.
\newblock {\em Rev. Mat. Complut.}, 35(1):133--146, 2022.

\end{thebibliography}
\bibliographystyle{abbrv}
\Addresses
\end{document}